\documentclass[a4paper]{amsart}
\pdfoutput=1

\usepackage[utf8]{inputenc}
\usepackage[T1]{fontenc}

\usepackage{amsthm, amssymb, amsmath, amsfonts, mathrsfs}

\usepackage{mathtools} 

\usepackage{MnSymbol}

\usepackage{scalerel} 

\usepackage[ocgcolorlinks, colorlinks=true, pdfstartview=FitV, linkcolor=blue, citecolor=blue, urlcolor=blue, pagebackref=false]{hyperref}

\usepackage{microtype}

\newtheorem{thm}{Theorem}[section]
\newtheorem{prop}[thm]{Proposition}
\newtheorem{lem}[thm]{Lemma}

\theoremstyle{remark}
\newtheorem{rem}[thm]{Remark}

\renewcommand{\le}{\leqslant}
\def\les{\lesssim}
\renewcommand{\ge}{\geqslant}
\renewcommand{\leq}{\leqslant}
\renewcommand{\geq}{\geqslant}
\renewcommand{\subset}{\subseteq}

\newcommand{\E}{\mathbb{E}}
\newcommand{\B}{\mathbb{B}}

\newcommand{\cE}{\mathcal{E}}
\newcommand{\cG}{\mathcal{G}}
\newcommand{\U}{\mathcal{U}}

\renewcommand{\L}{\mathscr{L}}
\newcommand{\K}{\mathscr{K}}

\newcommand{\Ll}{\left}
\newcommand{\Rr}{\right}

\newcommand{\R}{\mathbb{R}}
\newcommand{\C}{\mathcal{C}}

\newcommand{\Z}{\mathbb{Z}}
\renewcommand{\P}{\mathbb{P}}

\newcommand{\td}{\tilde}
\newcommand{\eps}{\varepsilon}
\def\d{{\mathrm{d}}}

\newcommand{\var}{\mathbb{V}\mathrm{ar}}

\newcommand{\h}{\mathsf{h}}
\newcommand{\loc}{\mathrm{loc}}

\newcommand{\Qh}{\mathsf{Q}}

\newcommand{\cu}{{\scaleobj{1.2}{\square}}}
\newcommand{\fint}{\strokedint}

\numberwithin{equation}{section}

\title[Fluctuations in stochastic homogenization]{Scaling limit of fluctuations in stochastic homogenization}
\author{Yu Gu, Jean-Christophe Mourrat}

\address[Yu Gu]{Department of Mathematics, Building 380, Stanford University, Stanford, CA, 94305, USA}

\address[Jean-Christophe Mourrat]{ENS Lyon, CNRS, 46 allée d'Italie, 69007 Lyon, France}

\begin{document}
\begin{abstract}

We investigate the global fluctuations of solutions to elliptic equations with random coefficients in the discrete setting. In dimension $d\geq 3$ and for i.i.d.\ coefficients, we show that after a suitable scaling, these fluctuations converge to a Gaussian field that locally resembles a (generalized) Gaussian free field. The paper begins with a heuristic derivation of the result, which can be read independently and was obtained jointly with Scott Armstrong.

\bigskip

\noindent \textsc{MSC 2010:} 35B27, 35J15, 35R60, 60G60.

\medskip

\noindent \textsc{Keywords:} quantitative homogenization, central limit theorem, Helffer-Sj\"ostrand representation, Stein's method.

\end{abstract}
\maketitle
%
%
%
%
%
%
%
%

\section{Heuristics}
\label{s:heu}

The goal of this paper is to give a precise description of the fluctuations of solutions of elliptic equations with random coefficients, in the large scale limit. Before stating our precise assumptions and results, we present powerful heuristics that enable to guess the results and give a better comprehension of the phenomena\footnote{The recording of a talk presenting this is also available at {\url{http://goo.gl/5bgfpR}}.}. These heuristics were obtained in collaboration with Scott Armstrong, whom we warmly thank for letting us include this material here.

\subsection{The (generalized) Gaussian free field}
We start by introducing white noise and (generalized) Gaussian free fields. These will be the fundamental objects used in the heuristic derivation of the large-scale behavior of the first-order correction to stochastic homogenization below.

The random distribution $w$ is a (one-dimensional) white noise with variance $\sigma^2$ if for every $\phi \in C^\infty_c(\R)$, $w(\phi)$ is a centered Gaussian random variable with variance $\sigma^2 \int \phi^2$. (We can in fact define $w(\phi)$ for any $\phi \in L^2(\R^d)$ by density.)
Informally,
$$
\E[w(x) \, w(y)] 
= \sigma^2 \, \delta(x-y),
$$
where $\delta$ is a Dirac mass at the origin.
More generally, the random, $d$-dimensional distribution $W = (W_1,\ldots, W_d)$ is a white noise with covariance matrix $\Qh$ if for every $\phi = (\phi_1, \ldots, \phi_d) \in C^\infty_c(\R^d)$, 
$$
W(\phi) := W_1(\phi_1) + \cdots + W_d(\phi_d)
$$
is a centered Gaussian random variable with variance $\int \phi \cdot \Qh \phi$. Informally, 
$$
\E[W_i(x) \, W_j(y) ] = \Qh_{ij} \, \delta(x-y).
$$

In dimension $d=1$, one way to define a Brownian motion $B$ is to ask it to satisfy
\begin{equation}
\label{e.brownian}
B' = w,
\end{equation}
where $w$ is a one-dimensional white noise, and $B'$ denotes the derivative of $B$. The Gaussian free field is a high-dimensional version of Brownian motion. 
By analogy with \eqref{e.brownian}, we may want to ask a Gaussian free field $\Phi$ to satisfy $\nabla \Phi = W$, where $W$ is a $d$-dimensional white noise. However, this does not make sense because $W$ is not a gradient field; so we will instead define $\nabla \Phi$ as the $L^2$ projection of $W$ onto the space of gradient fields. In view of the Helmholtz-Hodge decomposition of any vector field into a potential part and a solenoidal part: 
\begin{equation}
\label{e.helmholtz}
L^2 = \{\nabla u \} \stackrel{\perp}{\oplus} \{\mathbf g \, : \,  \nabla \cdot \mathbf g = 0\},
\end{equation}
this leads us to the equation
\begin{equation}
\label{e.GFF0}
-\Delta \Phi = \nabla \cdot W.
\end{equation}
A minor variant of the construction above is to consider the Helmholtz-Hodge projection with respect to a uniform background metric given by a symmetric matrix $a_\h$. In this case, equation \eqref{e.GFF0} becomes
\begin{equation}
\label{e.GFF}
-\nabla \cdot a_\h\nabla \Phi = \nabla \cdot W.
\end{equation}
Since the equation is linear, one can give a mathematically precise sense of $\nabla \Phi$ using classical arguments of the theory of distributions. (In dimension one, this also gives a sensible way to define a ``Brownian motion'' on the circle, i.e.\ a Brownian bridge, as opposed to \eqref{e.brownian} which does not satisfy the compatibility condition $\int w = 0$.) We take \eqref{e.GFF} as the definition of the Gaussian free field associated with $a_\h$ and $\Qh$. 

On the full space $\R^d$ with $d \ge 3$, one can also define $\Phi$ itself (and not only $\nabla \Phi$), e.g.\ as the limit as $\mu$ tends to $0$ of $\Phi_\mu$ such that
$$
(\mu - \nabla \cdot a_\h \nabla ) \Phi_\mu = \nabla \cdot W.
$$
In this case, one can express the two-point correlation function of $\Phi$ in terms of the Green function $\cG_\h$ of $- \nabla \cdot a_\h \nabla$ and the covariance matrix $\Qh$ of $W$:
\begin{equation}
\label{e.two-point}
\E[\Phi(0) \, \Phi(x)] = \int \nabla \cG_\h(y) \cdot \Qh \nabla \cG_\h(y-x) \, \d y.
\end{equation}
If $\Qh$ happens to be a multiple of $a_\h$, then an integration by parts enables to replace the integral above by a constant times $\cG_\h(x)$, and we recover the more common definition of the Gaussian free field as a Gaussian field whose covariance kernel is a Green function. However, for generic $a_\h$ and $\Qh$, the correlation in \eqref{e.two-point} cannot be expressed as a Green function. In other words, our definition of (generalized) Gaussian free field is wider than the standard one.

It is important for the remainder of the discussion to be familiar with the scaling and regularity properties of white noise and Gaussian free fields. As for white noise, $W(r \ \cdot \,)$ has the same law as $r^{-\frac d 2} W$. In particular, thinking of $r \to 0$, we see that zooming in on $W$ at scale $r$ produces a blow-up of $r^{-\frac d 2}$ (and conversely if we think of $r \to +\infty$). This is an indication of the fact that $W$ has (negative) H\"older regularity $\alpha$ for every $\alpha < -\frac d 2$, and no more. In view of \eqref{e.GFF}, the Gaussian free field $\Phi$ is such that $\Phi(r \ \cdot \,)$ has the same law as $r^{-\frac d 2 + 1} \, \Phi$, and has H\"older regularity $\alpha$ for every $\alpha < -\frac d 2 + 1$. In dimension $d = 1$, we recover the fact that Brownian motion trajectories have H\"older regularity $\alpha$ for every $\alpha < \frac 1 2$. In higher dimensions, the Gaussian free field fails to have regularity $0$; it only makes sense as a distribution, but not as a function.

\subsection{Homogenization and random fluctuations}

We now turn to the homogenization of the operator $-\nabla \cdot {a} \nabla$, where ${a}: \R^d \to \R^{d \times d}$ is a random field of symmetric matrices. We assume that the law of ${a}$ is stationary and posesses very strong mixing properties (e.g.\ finite range of dependence), and that $I_d \le {a} \le C I_d$ for some constant $C < \infty$. In this case, it is well-known that the large scale properties of the operator $-\nabla \cdot {a} \nabla$ resemble those of the homogeneous operator $-\nabla \cdot a_\h \nabla$, for some constant, deterministic matrix $a_\h$. Our goal is to describe the next-order correction. For $p \in \R^d$ and $\cu_r := (-\frac r 2, \frac r 2)^d \subset \R^d$, we introduce
\begin{equation}
\label{e.defnu}
\nu(\cu_r,p) := \inf_{v \in H^1_0(\cu_r)} \fint_{\cu_r} \frac 1 2 (p+ \nabla v) \cdot {a} (p+\nabla v)
\end{equation}
(where $\fint_{\cu_r} = |\cu_r|^{-1} \int_{\cu_r}$).
This quantity is subadditive: if $\cu_r$ is partitioned into subcubes $(y+\cu_{s})_y$, then $\nu(\cu_r,p)$ is smaller than the average over $y$ of $\nu(y+\cu_s,p)$. Indeed we can glue the minimizers of $\nu(y+\cu_s,p)$ and create a minimizer candidate for $\nu(\cu_r,p)$. Roughly speaking, it was shown in \cite{dal} that homogenization follows from the fact that
\begin{equation}
\label{e.dalma}
\nu(\cu_r,p) \xrightarrow[r \to \infty]{} \frac 1 2 p \cdot a_\h p
\end{equation}
(which itself is a consequence of the subadditve ergodic theorem). 
It is natural to expect the next-order correction to homogenization to follow from the understanding of the next-order correction to \eqref{e.dalma}. However, the next-order correction to \eqref{e.dalma} is driven by a boundary layer, which is of order $r^{-1}$ (see \cite{armstrong3}), and is not relevant to the understanding of the interior behavior of solutions. We thus assume that $\nu$ has been suitably modified into $\td \nu$ in order to get rid of the boundary layer. After performing this modification, we expect $\td \nu(\, \cdot \,, p)$ to be ``almost additive'' \cite{armstrong3}, and therefore that
$$
|\cu_r|^{1/2} \Ll( \td \nu(\cu_r,p) - \frac 1 2 p \cdot a_\h p \Rr) 
$$
converges to a Gaussian random variable as $r$ tends to infinity (as a consequence of the strong mixing assumption on the coefficients). Closely related statements were proved in \cite{nolen2011normal,biskup2014central,rossignol2012noise, nolen2014normal,GN2014CLT}. We want to encode this information in a way that is consistent with respect to changing the vector $p$, the scale $r$ and translations of the cube. For this purpose, we let $W$ be a matrix-valued white noise field such that as $r$ becomes large,
\begin{equation}
\label{e.key}
\td \nu(x + \cu_r,p) \simeq \frac 1 2 p \cdot \Ll(a_\h + W_r(x) \Rr) p ,
\end{equation}
where $W_r$ is the spatial average of $W$ on scale $r$:
\begin{equation}
\label{e.convol}
W_r (x) := \fint_{x + \cu_r} W.
\end{equation}
(We may also think of $W_r$ as the convolution of $W$ with a rescaled bump function: 
$W_r := W \star \chi^{(r)}$, with $\chi \in C^\infty_c(\R^d,\R_+)$ such that $\int \chi = 1$ and $\chi^{(r)} := r^{-d} \chi(\cdot/r)$.)
This encodes in particular the fact that $\nu(x+\cu_r,p)$ and $\nu(y+\cu_r,q)$ are asymptotically independent if $x+\cu_r$ and $y+\cu_r$ are disjoint. Recall that each coordinate of $W_r(x)$ is of order $|\cu_r|^{-1/2} = r^{-d/2} \ll 1$. We interpret \eqref{e.key} as indicating that if a function locally minimizes the energy over $x+\cu_r$ and has average gradient $p$, then its energy over $x + \cu_r$ is approximately $\frac 1  2 p \cdot (a_\h + W_r(x)) p$. 

The corrector for $-\nabla\cdot {a}\nabla$ in the direction $p$ is usually defined as the sublinear function solving
\[
-\nabla \cdot {a}(p+\nabla \phi)=0
\]
in the whole space. We think of $\phi$ as the minimizer in the definition of $\nu(\cu_R,p)$, for $R$ extremely large (in fact, infinite), and focus our attention on understanding the spatial average $\phi_r$ of $\phi$ on scale $r$, $ 1 \ll r \ll R$. The discussion above suggests that $\phi_r$ minimizes the \emph{coarsened} energy function
$$
v \mapsto \fint_{\cu_R} \frac 1 2 (p+\nabla v) \cdot (a_\h + W_r) (p+\nabla v),
$$
whose Euler-Lagrange equation is
$$
-\nabla \cdot (a_\h + W_r) (p+\nabla \phi_r) = 0.
$$
Rearranging, we obtain
$$
-\nabla \cdot (a_\h + W_r)\nabla \phi_r = \nabla  \cdot (W_r p).
$$
Since $W_r \ll 1$, we have $\nabla \phi_r \ll 1$. Therefore, the term $W_r$ on the left-hand side can be neglected, and we obtain the Gaussian-free-field equation
\begin{equation}
\label{e.corr-eq}
-\nabla \cdot a_\h \nabla \phi_r = \nabla \cdot (W_r p).
\end{equation}
We summarize this heuristic computation as follows.
\begin{itemize}
\item We let $W$ be the matrix-valued white noise field whose covariance is related to the fluctuations of the energy via \eqref{e.key};
\item we let $\Phi$ be the random distribution defined by 
\begin{equation}
\label{e.def-Phi}
-\nabla \cdot a_\h \nabla \Phi = \nabla \cdot (W p)\ ;
\end{equation}
\item then the large-scale spatial averages of the corrector $\phi$ have about the same law as those of $\Phi$. In other words (and by the scale invariance of $\Phi$), the random distribution $r^{\frac d 2 - 1} \phi(r \ \cdot \,)$ converges in law to $\Phi$ in a suitably weak topology, as $r$ tends to infinity.
\end{itemize}
A similar analysis can be performed for, say, solutions of equations of the form
\begin{equation}
\label{e.model-eq}
-\nabla \cdot {a} \nabla u = f \qquad \mbox{in $\R^d$, $d \ge 3$},
\end{equation}
where $f \in C^\infty_c(\R^d)$ varies on scale $\eps^{-1} \gg 1$. (The function $f$ should be of order $\eps^2$ in order for $u$ to be of order $1$.)  We consider the spatial average $u_r$ of $u$ over scale $r$, $1 \ll r \ll \eps^{-1}$. By the same reasoning as above, we expect $u_r$ to satisfy the coarsened equation
$$
-\nabla \cdot (a_\h+ W_r) \nabla u_r = f.
$$
We write $u_r = u_\h + \td u_r$, where $u_\h$ solves
$$
-\nabla \cdot a_\h \nabla u_\h = f,
$$
so that 
$$
-\nabla \cdot (a_\h + W_r) \nabla \td u_r = \nabla \cdot (W_r \nabla u_\h).
$$
As before, we expect the term $W_r$ on the left-hand side to be negligible, so we obtain
\begin{equation}
\label{e.fluct-eq}
-\nabla \cdot a_\h \nabla \td u_r = \nabla \cdot (W_r \nabla u_\h).
\end{equation}
We stress that if instead we use the formal two-scale expansion $u \simeq u_\h + \sum_i \phi^{(i)} \partial_i u_\h$ and the large-scale description of the corrector in \eqref{e.corr-eq}, we are led to a different and \emph{incorrect} result.

We can again summarize our conclusions as follows.
\begin{enumerate}
\item We define the matrix-valued white noise field as above, according to \eqref{e.key};
\item we let $\mathscr{U}$ be the random distribution defined by the equation
\begin{equation}
\label{e.defU}
-\nabla \cdot a_\h \nabla \mathscr{U} = \nabla \cdot (W \nabla u_\h) \ ;
\end{equation}
\item Then the large-scale averages of $u-u_\h$ have about the same law as those of~$\mathscr{U}$.
\end{enumerate}
The random distribution $\mathscr{U}$ is not a (generalized) Gaussian free field per se because the term $\nabla u_\h$ appearing in its defining equation varies over large scales. However, it will have similar small-scale features. If we normalize $f$ so that $u_\h$ is of order~$1$, then $\nabla u_\h$ is of order $\eps$, over a length scale of order $\eps^{-1}$. Hence, we can think of $\eps^{-1} \mathscr U$ as locally like a (generalized) Gaussian free field (of order $1$), and being close to $0$ outside of a domain of diameter of order $\eps^{-1}$. Incidentally, if we make the (unjustified) ansatz that $u-u_\h$ is approximately a regularization of $\mathscr{U}$ on the unit scale (that is, if we pretend that the conclusion (3) above actually holds on the unit scale), then we recover the (correct) error estimate
$$
\Ll(\eps^d \int |u-u_\h|^2 \Rr)^{1/2}\lesssim
\left|
\begin{array}{ll}
\eps \log^{1/2}(\eps^{-1}) & \text{if } d = 2, \\
\eps  & \text{if } d \ge 3.
\end{array}
\right.
$$

Mathematically speaking, a version of the statement that $r^{\frac d 2 - 1} \phi(r \ \cdot\,)$ converges to a (generalized) Gaussian free field was proved in \cite{mourrat2014correlation, MN}. The aim of the present work is to prove a version of the conclusion (3) above. 

\bigskip

Before turning to this, we want to emphasize why we believe these results to be of practical interest. 
Homogenization itself is interesting since it enables to describe approximations of solutions of equations with rapidly oscillating coefficients by solutions of simple equations described by a few effective parameters. What the above arguments show is that the same is true of the next-order correction: in order to describe it, it suffices to know the few parameters describing the covariance of the white noise~$W$. The white noise $W$ takes values in symmetric matrices, so its covariance matrix is fully described by $N(N+1)/2$ parameters, where $N = d(d+1)/2$ ($6$ parameters in dimension $2$, $21$ parameters in dimension $3$). Naturally, fewer parameters are necessary for problems with additional symmetries. For instance, the noise is described by only one parameter in any dimension if we assume rotational invariance.

\section{Introduction}

\subsection{Main result}

We focus on dimension $3$ and higher and on a discrete setting. Our main assumption is that the random coefficients are i.i.d.\ and bounded away from $0$ and infinity. Our goal is to justify the points (1-3) listed above in this context.

\medskip 

In order to state our assumptions and results more precisely, we introduce some notations.
We work on the graph $(\Z^d,\mathbb{B})$ with $d\geq 3$, where $\mathbb{B}$ is the set of nearest-neighbor edges. Let $e_1,\ldots,e_d$ be the canonical basis in $\Z^d$. For every edge $e\in \mathbb{B}$, there exists a unique pair $(\underline{e},i)\in \Z^d\times \{1,\ldots,d\}$ such that $e$ links $\underline{e}$ to $\underline{e}+e_i$. We will write $\bar{e}=\underline{e}+e_i$ and $e=(\underline{e},\bar{e})$. 

We give ourselves a family of i.i.d.\ random variables indexed by the edges of the graph. For convenience, we will assume that these random variables can be built from a family of i.i.d.\ standard Gaussians. More precisely, we let $\zeta = (\zeta_e)_{e \in \B} \in \Omega := \R^\B$ be i.i.d.\ standard Gaussian random variables, and $\eta$ be a smooth function that is bounded away from zero and infinity with bounded first and second derivatives. The \emph{conductance} associated with the edge $e$ is then defined to be $\eta(\zeta_e)$. The space $\Omega$ is equipped with the product Borel $\sigma$-algebra, and we denote the law of $\zeta$ by $\P$, with associated expectation $\E$. 

Let $\tilde{a}:\Z^d\times \Omega\to \R^{d\times d}$ be the matrix-valued function such that $\tilde{a}(x,\zeta)=\mathrm{diag}(\tilde{a}_1(x,\zeta),\ldots,\tilde{a}_d(x,\zeta))$ with $\tilde{a}_i(x,\zeta)=\eta(\zeta_{(x,x+e_i)})$. In what follows, we will most of the time keep the dependence on $\zeta$ implicit in the notation.

For any $f:\Z^d\to \R$ we define the discrete gradient $\nabla f:=(\nabla_1f,\ldots,\nabla_df)$ by $\nabla_i f(x)=f(x+e_i)-f(x)$. For any $g:\Z^d\to \R^d$ we define the discrete divergence $\nabla^*g:=\sum_{i=1}^d \nabla_i^*g_i$ by $\nabla_i^* g_i(x)=g_i(x-e_i)-g_i(x)$. We define $\nabla_\eps, \nabla_\eps^*$ similarly for functions defined on $\eps\Z^d$, i.e., $\nabla_{\eps,i}h(x)=\eps^{-1}(h(x+\eps e_i)-h(x))$ and $\nabla_{\eps,i}^*h(x)=\eps^{-1}(h(x-\eps e_i)-h(x))$ for $h:\eps\Z^d\to \R$.

For any $\eps>0$, we consider the following elliptic equation with a slowly varying source term:
\begin{equation}
\nabla^* \tilde{a}(x)\nabla u(x)=f(\eps x) \qquad (x\in \Z^d),
\label{mainEq}
\end{equation}
where $f\in\C_c(\R^d)$ is compactly supported and continuous. The unique solution to \eqref{mainEq} that decays to zero at infinity is given by  
\begin{equation*}
u(x)=\sum_{y\in \Z^d} G(x,y)f(\eps y),
\end{equation*}
with $G(x,y)$ the Green function of $\nabla^*\tilde{a}(x)\nabla$ (recall that the dependence on $\zeta$ is kept implicit, so $u$ and $G$ are random).

We define $u_\eps(x)=\eps^2u(\frac{x}{\eps})$, which solves
\begin{equation*}
\nabla^*_\eps \tilde{a}(\frac{x}{\eps}) \nabla_\eps u_\eps(x)=f(x) \qquad (x\in \eps \Z^d),
\end{equation*}
and its limit $u_\h$, the solution of the \emph{homogenized} equation in continuous space:
\begin{equation*}
-\nabla \cdot a_\h\nabla u_\h(x)=f(x) \qquad (x\in \R^d),
\end{equation*}
where the \emph{homogenized matrix} $a_\h$ is deterministic and constant in space. 

We are interested in the random fluctuations of $u_\eps$ after a spatial average. In other words, we think of $u_\eps$ as a (random) distribution $\U_\eps^{(f)}$, which acts on a test function $g\in \C_c(\R^d)$ as
\begin{equation}
\U_\eps^{(f)}(g):=\eps^d\sum_{x\in \eps\Z^d} u_\eps(x)g(x)=\eps^{d+2}\sum_{x,y\in \Z^d}G(x,y)g(\eps x)f(\eps y).
\label{glFlu}
\end{equation}

Following \cite{mourrat2015tight}, for any $\alpha < 0$, we denote by $\C^\alpha_\loc = \C^\alpha_\loc(\R^d)$ the (separable) local H\"older space of regularity $\alpha$. Here is our main result.
\begin{thm}
Recall that $d\geq 3$. We define the random distribution $\mathscr{U}_\eps^{(f)}$ by
\begin{equation*}
\mathscr{U}_\eps^{(f)}(g)=\eps^{-\frac{d}{2}}(\U_\eps^{(f)}(g)-\E\{\U_\eps^{(f)}(g)\}).
\end{equation*}
For every $\alpha<1-\frac{d}{2}$, the distribution $\mathscr{U}_\eps^{(f)}$ converges in law to $\mathscr{U}^{(f)}$ as $\eps\to 0$ for the topology of $\C_{\loc}^\alpha$, where $\mathscr{U}^{(f)}$ is the Gaussian random field such that for $g\in\C_c(\R^d)$, $\mathscr{U}^{(f)}(g)$ is a centered Gaussian with variance
\begin{equation}
\sigma_g^2=\int_{\R^{2d}}\mathscr{K}_f(x,z)g(x)g(z)dxdz
\label{mainVar}
\end{equation}
with $\mathscr{K}_f$ given explicitly by \eqref{gauK}.
\label{t:mainTh}
\end{thm}

\begin{rem}
\label{re:sospde}
We can write the solution to \eqref{e.defU} formally as
\[
\begin{aligned}
\mathscr{U}(x)&=\int_{\R^d}\cG_\h(x-y)\nabla\cdot (W(y)\nabla u_\h(y))dy\\
&=\int_{\R^{2d}}\nabla \cG_\h(x-y) \cdot W(y)\nabla \cG_\h(y-z)f(z)dzdy,
\end{aligned}
\]
where $\cG_\h$ is the Green function associated with $-\nabla\cdot a_\h \nabla$. By a straightforward calculation, this random distribution tested with $g$ has a Gaussian distribution with mean zero and variance given by
\[
\E\{|\int_{\R^d}\mathscr{U}(x)g(x)dx|^2\}=\int_{\R^{2d}}\mathscr{K}_f(x,z)g(x)g(z)dxdz,
\]
provided the covariance of the white noise $W$ is given by $\{\tilde{K}_{ijkl}\}$ that appears in the definition of $\mathscr{K}_f(x,z)$. In other words, the limiting distribution $\mathscr{U}^{(f)}$ obtained in Theorem~\ref{t:mainTh} can be represented as the solution of \eqref{e.defU}.
\end{rem}

\begin{rem}
Theorem~\ref{t:mainTh} implies the joint convergence in law of $(\mathscr{U}_\eps^{(f_1)}, \ldots, \mathscr{U}_\eps^{(f_k)})$ to a Gaussian vector field whose covariance structure is obtained by polarization of $f \mapsto \mathscr{K}_f$. Indeed, this follows from the fact that Theorem~\ref{t:mainTh} gives the scaling limit of any linear combination of $\mathscr{U}_\eps^{(f_1)}, \ldots, \mathscr{U}_\eps^{(f_k)}$, by linearity of $f \mapsto \mathscr{U}_\eps^{(f)}$.
\end{rem}

\begin{rem}
A similar result is proved in \cite{gu2015} when $d=1$, using a different method. In this case, $\mathscr{U}$ in \eqref{e.defU} has H\"older regularity $\alpha$ for every $\alpha<\frac12$, so it makes sense as a function. We expect the result to hold for $d=2$ as well, but our method would have to be modified to handle the fact that only the gradient of $u$ is really defined by \eqref{mainEq} in this case.
\end{rem}

Theorem~\ref{t:mainTh} is a consequence of the following two propositions.

\begin{prop}
For every $g\in\C_c(\R^d)$, $\mathscr{U}_\eps^{(f)}(g)$ converges in law to $\mathscr{U}^{(f)}(g)$ as $\eps\to 0$.
\label{p:conG}
\end{prop}

\begin{rem}
In fact, when $\sigma_g^2 \neq 0$, our proof gives a rate of convergence of $\mathscr{U}_\eps^{(f)}(g)$ to $\mathscr{U}^{(f)}(g)$, see Remark~\ref{r:rate} below.
\end{rem}

\begin{prop}
For every $\alpha<1-\frac{d}{2}$, $\mathscr{U}_\eps^{(f)}$ is tight in $\C_{\loc}^\alpha$.
\label{p:tight}
\end{prop}

\subsection{Context}

Stochastic homogenization of divergence form operators started from the work of Kozlov \cite{kozlov1979averaging} and Papanicolaou-Varadhan \cite{papanicolaou1979boundary}, where a \emph{qualitative} convergence of heterogeneous random operators to homogeneous deterministic ones is proved. \emph{Quantitative} aspects were explored as early as in \cite{yurinskii1986averaging}. However, optimal bounds on the size of errors were obtained only recently in a series of papers \cite{gloria2011optimal,gloria2012optimal, mourrat2012kantorovich, gloria2013quantification, gloria2014optimal, gloria2014quantitative,gloria2014improved}. Regularity estimates that are optimal in terms of stochastic integrability have been worked out in \cite{armstrong1,armstrong2,gloria2014optimal}.

Our focus in this paper is to go beyond estimating the size of the errors, and understand the probability law of the rescaled random fluctuations. In this direction, central limit theorems for approximations of homogenized coefficients are obtained in \cite{nolen2011normal,biskup2014central,rossignol2012noise, nolen2014normal,GN2014CLT}. The scaling limit of the corrector is investigated in \cite{mourrat2014correlation,MN}. In the continuous setting, \cite{gu-mourrat} indicates that when $d\geq 3$, the corrector should capture the first order fluctuation of the heterogeneous solution in a pointwise sense, but it is not clear whether it captures the fluctuations of the solution after a spatial average. A surprising feature of our result is that the limiting fluctuations are \emph{not} those induced by the corrector alone.

Our approach is based on that of \cite{mourrat2014correlation,MN}. The fact that $\mathscr{U}_\eps^{(f)}(g)$ divided by its standard deviation converges to a standard Gaussian is derived using a second order Poincar\'e inequality developed by Chatterjee \cite{chatterjee2009fluctuations}, in the spirit of Stein's method. (We will in fact use a slightly more convenient form of this result derived in \cite{MN}.) The main difficulty lies in the proof of the convergence of the variance of $\mathscr{U}_\eps^{(f)}(g)$. A Helffer-Sj\"ostrand formula enables to rewrite this variance in terms of gradients of the Green function. A quantitative two-scale expansion for the gradient of the Green function was worked out in \cite{mourrat2014correlation}. Here, we follow the idea of \cite{gloria2014regularity} of introducing a stationary skew-symmetric tensor, which is denoted by $\{\sigma_{ijk}\}_{i,j,k=1}^d$ and relates to the flux (see Lemma~\ref{l:geneCor}). In the language of differential forms, the flux in the $i$-th direction is a co-closed $1$-form, and we represent it as the co-differential of the $2$-form~$\sigma_i$. This object enables us to represent the error in the two-scale expansion in divergence form, and thus significantly improve the two-scale expansion and simplify the subsequent analysis.

\subsection{Organization of the paper}

The rest of the paper is organized as follows. We introduce basic notation and recall key estimates on correctors and Green functions in Section~\ref{s:stp}. Then we present some key ingredients in proving Theorem~\ref{t:mainTh} in Section~\ref{s:igd}, including the Helffer-Sj\"ostrand covariance representation, a quantitative two-scale expansion of the Green function and a second order Poincar\'e inequality. The proofs of Propositions~\ref{p:conG} and \ref{p:tight} are contained in Sections~\ref{s:cv}, \ref{s:cg} and \ref{s:tt}. Technical lemmas are left in the appendix.

\section{Setup}
\label{s:stp}

\subsection{Asymptotic variance}
For $x\in \Z^d$, we define the shift operator $\tau_x$ on $\Omega$ by $(\tau_x \zeta)_e=\zeta_{x+e}$, where $x+e:=(x+\underline{e},x+\bar{e})$ is the edge obtained by shifting $e$ by $x$. Since $\{\zeta_e\}_{e\in\mathbb{B}}$ are i.i.d., $\{\tau_x\}_{x\in \Z^d}$ is a group of measure-preserving transformations. With any measurable function $f:\Omega\to \R$, we can associate a stationary random field $\tilde{f}(x,\zeta)$ defined by
\begin{equation}
\tilde{f}(x,\zeta)=(T_xf)(\zeta)=f(\tau_x\zeta).
\end{equation}
The generators of $T_x$, denoted by $\{D_i\}_{i=1}^d$, are defined by $D_if:=T_{e_i}f-f$. The adjoint $D_i^*$ is defined by $D_i^* f:=T_{-e_i}f-f$. We denote the gradient on $\Omega$ by $D=(D_1,\ldots,D_d)$ and the divergence $D^*g:=\sum_{i=1}^d D_i^*g_i$ for $g:\Omega\to \R^d$. The inner product in $L^2(\Omega)$ and norm in $L^p(\Omega)$ are denoted by $\langle\cdot,\cdot\rangle$ and $\|\cdot\|_p$ respectively.

Most of the time, we keep the dependence on $\zeta$ implicit and write $\tilde{f}(x)=\tilde{f}(x,\zeta)=f(\tau_x\zeta)$. For any $e\in\mathbb{B}$, the discrete derivative on $e$ is defined by $\nabla\tilde{f}(e):=\tilde{f}(\bar{e})-\tilde{f}(\underline{e})$. If $e$ is in the $i-$th direction, we define $\xi(e):=\xi_i$ for any $\xi\in \R^d$, as the projection of $\xi$ onto $e$. 

For the random coefficients appearing in \eqref{mainEq}, we can define
\begin{equation*}
a(\zeta)=\mathrm{diag}(a_1(\zeta),\ldots,a_d(\zeta)):=\mathrm{diag}(\eta(\zeta_{e_1}),\ldots,\eta(\zeta_{e_d}))
\end{equation*}
so that $\tilde{a}(x,\zeta)=a(\tau_x\zeta)=\mathrm{diag}(\eta(\zeta_{x+e_1}),\ldots,\eta(\zeta_{x+e_d}))$.  Note that we also used $e_1,\ldots,e_d$ to denote the corresponding edges $(0,e_1),\ldots,(0,e_d)$. Recall that we assume that $C^{-1}<\eta<C$ and $|\eta'|,|\eta''|<C $ for some $C < \infty$. For simplicity we will henceforth write $a_e=\eta(\zeta_e)$.

Under the above assumptions, it is well-known that there exists a constant matrix $a_\h$ such that the operator $\nabla^*\tilde{a}\nabla$ homogenizes over large scale to the continuous operator $-\nabla\cdot a_\h\nabla$, the Green function of which we denote as $\cG_\h$.

One of the main ingredients in the analysis of stochastic homogenization is the so-called \emph{corrector}. For any fixed $\xi\in \R^d$ and $\lambda>0$, the regularized corrector $\phi_{\lambda,\xi}$ is defined through the following equation on probability space:
\begin{equation}
\lambda\phi_{\lambda,\xi}+D^*a(D\phi_{\lambda,\xi}+\xi)=0.
\label{mascor}
\end{equation}
It is proved in \cite{gloria2011optimal} that as $\lambda\to0$, $\phi_{\lambda,\xi}\to \phi_\xi$ in $L^2(\Omega)$, i.e., a stationary corrector $\phi_\xi$ exists such that 
\begin{equation}
D^*a(D\phi_\xi+\xi)=0.
\label{cor}
\end{equation} 
For $i=1,\ldots,d$, we will write $\phi_i=\phi_{e_i}$ and $\phi_{\lambda,i}=\phi_{\lambda,e_i}$. The homogenized matrix $a_\h$ is given by 
\begin{equation}
\xi^Ta_\h\xi=\E\{(\xi+D\phi_\xi)^Ta(\xi+D\phi_\xi)\}.
\label{effcst}
\end{equation}
 In the context of i.i.d.\ randomness, we have $a_\h=\bar{a}I_d$ for some constant $\bar{a}$, where $I_d$ is the identity matrix.

For a random variable $F\in L^2(\Omega)$, we say that $U=\partial_e F\in L^2(\Omega)$ is the weak derivative of $F$ with respect to $\zeta_e$ if the following holds: for any finite subset $\Lambda\in \B$ and any smooth, compactly supported function $G:\R^{|\Lambda|}\to \R$, we have
\begin{equation}
\E\{U G(\zeta)\}=\E\{F \zeta_e G(\zeta)\}-\E\{F\frac{\partial G}{\partial \zeta_e}(\zeta)\},
\label{eq:defwd}
\end{equation}
where $G(\zeta)$ depends only on $\{\zeta_{e'}\}_{e'\in\Lambda}$. We also call $\partial_e$ the \emph{vertical derivative},
and by \eqref{eq:defwd}, its adjoint (under the Gaussian measure) is
\begin{equation*}
\partial_e^*=-\partial_e+\zeta_e.
\end{equation*}
We define $\partial f:=(\partial_ef)_{e\in\mathbb{B}}$ and for $F=(F_e)_{e\in\mathbb{B}}$, $\partial^* F:=\sum_{e\in\mathbb{B}}\partial_e^*F_e$. The vertical Laplacian on the probability space is then defined by
\begin{equation*}
\L:=\partial^*\partial.
\end{equation*}

For $i,j,k,l=1,\ldots,d$, define
\begin{equation*}
K_{ijkl}(e):=\langle \partial_e a_e(e_i+\nabla\tilde{\phi}_i)(e)(e_j+\nabla\tilde{\phi}_j)(e),(1+\L)^{-1}\partial_e a_e(e_k+\nabla\tilde{\phi}_k)(e)(e_l+\nabla\tilde{\phi}_l)(e)\rangle,
\end{equation*}
and
\begin{equation}
\tilde{K}_{ijkl}:=\sum_{n=1}^d K_{ijkl}(e_n).
\label{excos}
\end{equation} 
The kernel $\K_f(x,z)$ appearing in Theorem~\ref{t:mainTh} is given by 
\begin{equation}
\begin{aligned}
\K_f(x,z)=\sum_{i,j,k,l=1}^d \tilde{K}_{ijkl}\int_{\R^{3d}}&\partial_{x_i}\cG_\h(v-x)\partial_{x_j}\cG_\h(v-y)\partial_{x_k}\cG_\h(v-z)\partial_{x_l}\cG_\h(v-w)\\
&f(y)f(w)dydwdv.
\label{gauK}
\end{aligned}
\end{equation}

\subsection{A comparison with the two-scale expansion}
Let us see that the global fluctuations of $u_\eps$ are \emph{not} those suggested by its two-scale expansion. Recall that $u_\eps$ and $u_\h$ satisfy
\begin{equation*}
\nabla^*_\eps \tilde{a}(\frac{x}{\eps}) \nabla_\eps u_\eps(x)=f(x) \qquad (x\in \eps \Z^d),
\end{equation*}
and 
\begin{equation*}
-\nabla\cdot a_\h\nabla u_\h(x)=f(x) \qquad (x\in\R^d).
\end{equation*}
A formal two-scale expansion gives
\begin{equation}
u_\eps(x)=u_\h(x)+\eps \nabla u_\h(x)\cdot \phi(\frac{x}{\eps})+o(\eps),
\label{2scale}
\end{equation}
where $\phi=(\phi_1,\ldots,\phi_d)$. In the continuous setting, \eqref{2scale} is proved rigorously with $o(\eps)/\eps\to 0$ in $L^1(\Omega)$ for fixed $x$ \cite[Theorem 2.3]{gu-mourrat}, i.e., the first order correction is indeed given by the corrector in a pointwise sense. Since $\phi$ is centered, we have a similar expansion for the random fluctuation, i.e.,
\begin{equation*}
u_\eps(x)=\E\{u_\eps(x)\}+\eps \nabla u_\h(x)\cdot \phi(\frac{x}{\eps})+o(\eps).
\end{equation*}
Concerning global fluctuations, we need to compare the random field $u_\eps(x)-\E\{u_\eps(x)\}$ with $\eps \nabla u_\h(x)\cdot \phi(\frac{x}{\eps})$. For a test function $g\in\C_c(\R^d)$, Theorem~\ref{t:mainTh} shows 
\begin{equation}
\eps^{\frac{d}{2}}\sum_{x\in\eps \Z^d}(u_\eps(x)-\E\{u_\eps(x)\})g(x)\Rightarrow N(0,\sigma_g^2),
\end{equation}
with $\sigma_g^2=\int_{\R^{2d}}\K_f(x,z)g(x)g(z)dxdz$. By \cite[Theorem~1.1]{MN}, we have
\begin{equation*}
\eps^{\frac{d}{2}}\sum_{x\in \eps\Z^d}\eps \nabla u_\h(x)\cdot \phi(\frac{x}{\eps})g(x)\Rightarrow N(0,\tilde{\sigma}_g^2),
\end{equation*}
with $\tilde{\sigma}_g^2=\int_{\R^{2d}}\tilde{\K}_f(x,z)g(x)g(z)dxdz$ and 
\begin{equation}
\begin{aligned}
\tilde{\K}_f(x,z)=\sum_{i,j,k,l=1}^d \tilde{K}_{ijkl}\int_{\R^{3d}}&\partial_{x_i}\cG_\h(v-x)\partial_{x_j}\cG_\h(x-y)\partial_{x_k}\cG_\h(v-z)\partial_{x_l}\cG_\h(z-w)\\
&f(y)f(w)dydwdv.
\label{gauKc}
\end{aligned}
\end{equation}
If $\sigma_g^2$ and $\td \sigma_g^2$ were equal for every admissible $f$ and $g$, then
$$
\sum_{i,j,k,l=1}^d \tilde{K}_{ijkl}\int_{\R^{d}}\partial_{x_i}\cG_\h(v-x)\partial_{x_j}\cG_\h(v-y)\partial_{x_k}\cG_\h(v-z)\partial_{x_l}\cG_\h(v-w) dv
$$
and 
$$
\sum_{i,j,k,l=1}^d \tilde{K}_{ijkl}
\int_{\R^{d}}\partial_{x_i}\cG_\h(v-x)\partial_{x_j}\cG_\h(x-y)\partial_{x_k}\cG_\h(v-z)\partial_{x_l}\cG_\h(z-w) dv
$$
would have to be equal almost everywhere (as functions of $x,y,z$ and $w$). However, the first quantity diverges when $y$ gets close to $w$ since $d\geq 3$, while this is not so for the second quantity. This shows that the fluctuations of $u_\eps$ are not those suggested by the two-scale expansion.

The heuristics in Section~\ref{s:heu} provide a clear picture of the above phenomenon (other than the explanation that $o(\eps)$ may contribute on the level of $\eps^{\frac{d}{2}}$ when $d\geq 3$). If we write the solution to \eqref{e.def-Phi} as $\Phi(x)=\int_{\R^d}\nabla \cG_\h(x-y)\cdot W(y)p dy$,
the rescaled limit of the corrector $\eps \nabla u_\h(x)\cdot \phi(\frac{x}{\eps})$ is
\begin{equation}
\eps^{-\frac{d}{2}+1} \nabla u_\h(x)\cdot \phi(\frac{x}{\eps})\to \int_{\R^d}\nabla \cG_\h(x-y)\cdot W(y)\nabla u_\h(x)dy,
\label{eq:limitcor}
\end{equation}
and we already know from Remark~\ref{re:sospde} that the rescaled limit of $u_\eps(x)-\E\{u_\eps(x)\}$ is
\begin{equation}
\eps^{-\frac{d}{2}}(u_\eps(x)-\E\{u_\eps(x)\})\to \int_{\R^d}\nabla \cG_\h(x-y)\cdot W(y)\nabla u_\h(y)dy.
\label{eq:limitso}
\end{equation}
Comparing the r.h.s. of \eqref{eq:limitcor} and \eqref{eq:limitso}, it is clear that they are two different Gaussian random fields. We further observe that they are linked through a Taylor expansion of $\nabla u_\h(y)$ around $x$. By writing 
\[
\nabla u_\h(y)=\nabla u_\h(x)+\nabla^2 u_\h(x)(y-x)+\ldots,
\]
with $\nabla^2 u_\h(x)$ the Hessian of $u_\h$, we have
\begin{equation}
\begin{aligned}
 &\int_{\R^d}\nabla \cG_\h(x-y)\cdot W(y)\nabla u_\h(y)dy\\
 =& \int_{\R^d}\nabla \cG_\h(x-y)\cdot W(y)\nabla u_\h(x)dy\\
 & +\int_{\R^d}\nabla \cG_\h(x-y)\cdot W(y)\nabla^2 u_\h(x)(y-x)dy+\ldots.
\end{aligned}
 \label{eq:exdis}
 \end{equation}
 The term $ \int_{\R^d}\nabla \cG_\h(x-y)\cdot W(y)\nabla^2 u_\h(x)(y-x)dy$ should correspond to the second order corrector obtained by the two-scale expansion, and we also expect those higher order terms appearing in \eqref{eq:exdis} to correspond to the rescaled limit of the higher order correctors (provided that they are stationary). It does not seem possible to simply add a finite number of terms in the two-scale expansion to recover the correct limiting field.

\subsection{Properties of correctors and Green functions}

We summarize here several results obtained in \cite{gloria2011optimal,marahrens2013annealed} which will be used frequently throughout the paper. Let $|x|$ be the norm of $x\in \Z^d$, and $|x|_*=2+|x|$.

\begin{prop}[Existence of stationary corrector and moment bounds \cite{gloria2011optimal}]
Recall that we assume $d \ge 3$. For every $\lambda > 0$, there exists a unique stationary solution $\phi_{\lambda,\xi}$ to equation \eqref{mascor}. Moreover, for every $p \ge 1$, $\E\{|\phi_{\lambda,\xi}|^p\}$ and $\E\{|D \phi_{\lambda,\xi}|^p\}$ are uniformly bounded in $\lambda > 0$. The limit $\phi_\xi = \lim_{\lambda \to 0} \phi_{\lambda,\xi}$ is well-defined in $L^p(\Omega)$ and is the unique centered stationary solution to \eqref{cor}.
\label{p:esCo}
\end{prop}

Denote by $G_{\lambda}(x,y)$ the Green function of $\lambda+\nabla^*\tilde{a}(x)\nabla$ (the dependence on the randomness $\zeta$ is kept implicit) and recall that $G(x,y)=G_0(x,y)$. The following pointwise bound holds:
\begin{equation*}
G_\lambda(x,y)\leq \frac{C}{|x|_*^{d-2}}e^{-c\sqrt{\lambda}|x|}
\end{equation*}
for some $c,C>0$. The following result controls the derivatives in the annealed sense.

\begin{prop}[annealed estimates on the gradients of the Green function \cite{marahrens2013annealed}]
For every $1\le p <\infty$, there exists $C_p<\infty$ such that for every $\lambda\geq 0$  and every $e,e' \in \mathbb{B}$,
$$
\|\nabla G_\lambda(0,e) \|_p\le \frac{C_p}{|\underline{e}|_*^{d-1}},
$$
$$
\| \nabla \nabla G_\lambda(e,e') \|_p\le \frac{C_p}{|\underline{e'}-\underline{e}|_*^{d}}.
$$
\label{p:esGr}
\end{prop}
\begin{rem}
Notice that $\nabla G(x,e)$ (for $x \in \Z^d$ and $e \in \mathbb{B}$) denotes the gradient of $G(x,\cdot)$ evaluated at the edge $e$. Similarly, $\nabla \nabla G(e,e')$ denotes the gradient of $\nabla G(\cdot,e')$ evaluated at the edge $e$.
\end{rem}

\subsection{Notation}
We summarize and introduce some more notations used throughout the paper.
\begin{itemize}
\item For $i=1,\ldots,d$, $e\in\B$ and $\tilde{f}$, $\nabla_i \tilde{f}(e):=\nabla_i\tilde{f}(\underline{e})$. Recall that without any subscript, $\nabla \tilde{f}(e)=\tilde{f}(\bar{e})-\tilde{f}(\underline{e})$, and for $x\in \Z^d$, $\nabla_i\tilde{f}(x)=f(x+e_i)-f(x)$ and $\nabla_i^*\tilde{f}(x)=f(x-e_i)-f(x)$.
\item We write $a\les b$ when $a\leq Cb$ for some constant $C$ independent of $\eps,e,x$.
\item For $a,b,c>0$, we write $a\les \frac{1}{b^{c-}}$ if for any $\delta>0$, there exists $C_\delta>0$ such that $a\leq C_\delta \frac{1}{b^{c-\delta}}$. In this way we have 
\begin{equation*}
\frac{\log |x|_*}{|x|_*^c}\les \frac{1}{|x|_*^{c-}}.
\end{equation*}
\item The Laplacian on $\Z^d$ and the \emph{horizontal} Laplacian on the probability space are both denoted by $\Delta=-\nabla^*\nabla$ and $\Delta=-D^*D$.
\item For a random environment $\zeta$ and edge $e\in \B$, we obtain the environment perturbed at $e$ by replacing $\zeta_e$ with an independent copy $\zeta_e'$ without changing other components $(\zeta_{e'})_{e'\neq e}$. The resulting new environment is denoted by $\zeta^e$.
\item For a random variable $f$ and an edge $e\in\B$, the variable perturbed at $e$ is denoted by $f^e(\zeta):=f(\zeta^e)$. For a stationary random field $\tilde{f}(x)=f(\tau_x\zeta)$, the field perturbed at $e$ is denoted by $\tilde{f}^e(x):=f(\tau_x\zeta^e)$.
\item The discrete homogenized Green function of $\lambda+\nabla^*a_\h\nabla$ is denoted by $G_{\h,\lambda}(x,y)$ for $\lambda\geq 0$, and $G_\h(x,y)=G_{\h,0}(x,y)$. Recall that the heterogeneous Green function of $\lambda+\nabla^*\tilde{a}\nabla$ is denoted by $G_\lambda(x,y)$, and that $G(x,y)=G_0(x,y)$. The continuous homogenized Green function of $-\nabla\cdot a_\h\nabla$ is $\cG_\h(x,y)$.
\item $\{e_i,i=1,\ldots,d\}$ represents the canonical basis of $\Z^d$, the corresponding edges, and the column vectors so that the identity matrix $I_d=[e_1,\ldots,e_d]$.
\item For functions of two variables, e.g., $G(x,y)$ with $x,y\in \Z^d$, we use $\nabla_{x,i},\nabla_{y,i}$ to denote the derivative with respect to $x_i,y_i$ respectively.
\item The arrow $\Rightarrow$ stands for convergence in law, and $N(0,\sigma^2)$ is the Gaussian law with mean $0$ and variance $\sigma^2$.
\item $a\vee b=\max(a,b)$ and $a\wedge b=\min(a,b)$.
\end{itemize}

\section{Helffer-Sj\"ostrand representation, two-scale expansion of the Green function, and second order Poincar\'e inequality}
\label{s:igd}

We divide the proof of Proposition~\ref{p:conG} into two steps. First, we show that 
\begin{equation}
\eps^{-d}\var\{\U_\eps^{(f)}(g)\}\to \sigma_g^2,
\label{conV}
\end{equation}
with $\sigma_g^2$ defined in \eqref{mainVar}. If $\sigma_g^2=0$, we conclude $\mathscr{U}_\eps^{(f)}(g)\to 0$ in $L^2(\Omega)$. Next we assume $\sigma_g^2>0$ and show  
\begin{equation}
\frac{\U_\eps^{(f)}(g)-\E\{\U_\eps^{(f)}(g)\}}{\sqrt{\var\{\U_\eps^{(f)}(g)\}}}\Rightarrow N(0,1).
\label{conN}
\end{equation}
Once this is done, we can write
\begin{equation*}
\mathscr{U}_\eps^{(f)}(g)=\frac{\U_\eps^{(f)}(g)-\E\{\U_\eps^{(f)}(g)\}}{\sqrt{\var\{\U_\eps^{(f)}(g)\}}}\times \sqrt{\eps^{-d}\var\{\U_\eps^{(f)}(g)\}}
\end{equation*}
to conclude that $\mathscr{U}_\eps^{(f)}(g)\Rightarrow N(0,\sigma_g^2)$.

The proof of \eqref{conV} uses the Helffer-Sj\"ostrand representation and a two-scale expansion of the Green function, while the proof of \eqref{conN} relies the second order Poincar\'e inequality developed by Chatterjee \cite{chatterjee2009fluctuations} and revisited in \cite{MN}. Both of them require taking vertical derivatives of $\U_\eps^{(f)}(g)$ with respect to the underlying Gaussian variables $\zeta_e$. Recall that $\U_\eps^{(f)}(g)=\eps^{d+2}\sum_{x,y\in\Z^d}G(x,y)g(\eps x)f(\eps y)$ is a finite linear combination of $G(x,y)$. By Lemma~\ref{l:vd}, we have
\begin{equation*}
\partial_e \U_\eps^{(f)}(g)=-\eps^{d+2}\sum_{x,y\in\Z^d}\partial_e a_e\nabla G(x,e)\nabla G(y,e)g(\eps x)f(\eps y).
\end{equation*}

We introduce the key elements in proving \eqref{conV} and \eqref{conN} in the following section.

\subsection{Helffer-Sj\"ostrand representation and a two-scale expansion of the Green function}

\begin{prop}[Helffer-Sj\"ostrand representation \cite{mourrat2014correlation}]
Let $f,g:\Omega\to \R$ be centered square-integrable functions such that for every $e\in \mathbb{B}$, $\partial_e f,\partial_e g\in L^2(\Omega)$. We have
\begin{equation*}
\langle f,g\rangle=\sum_{e\in \mathbb{B}}\langle \partial_e f,(1+\L)^{-1}\partial_e g\rangle.
\end{equation*}
Moreover, for every $p\geq 2$, $(1+\L)^{-1}$ is a contraction from $L^p(\Omega)$ to $L^p(\Omega)$.
\label{p:HSre}
\end{prop}

Since $$\U_\eps^{(f)}(g)=\eps^{d+2}\sum_{x,y\in \Z^d}G(x,y)g(\eps x)f(\eps y),$$ the proof of \eqref{conV} is reduced to asymptotics of $\mathrm{Cov}\{G(x,y),G(z,w)\}$ when the mutual distances between $x,y,z$ and $w$ are large. By applying Proposition~\ref{p:HSre} and Lemma~\ref{l:vd}, the covariance is given by
\begin{equation*}
\begin{aligned}
\mathrm{Cov}\{G(x,y),G(z,w)\}=&\sum_{e\in \mathbb{B}}\langle \partial_e G(x,y),(1+\L)^{-1}\partial_e G(z,w)\rangle\\
=&\sum_{e\in \mathbb{B}}\langle \partial_e a_e\nabla G(x,e)\nabla G(y,e),(1+\L)^{-1}\partial_e a_e \nabla G(z,e)\nabla G(w,e)\rangle.
\end{aligned}
\end{equation*}
To prove the asymptotics, we need an expansion of $\nabla G(x,e)$. 
The following proposition is our main result in this section and one of the main ingredients to prove~\eqref{conV}.

\begin{prop}
Recall that $G$ and $G_\h$ are the Green functions of $\nabla^*\tilde{a}\nabla$ and $\nabla^*a_\h\nabla$ respectively. For any $e\in \B$, we have
\begin{equation*}
\|\nabla G(0,e)-\nabla G_\h(e)-\sum_{k=1}^d \nabla_k G_\h(e)\nabla\tilde{\phi}_k(e)\|_2\les\frac{\log |\underline{e}|_*}{|\underline{e}|_*^d}.
\label{2scaleGr}
\end{equation*}
\label{p:2scaleGr}
\end{prop}

An immediate consequence is that for any $X\in L^2(\Omega)$,
\begin{equation}
|\langle X, \nabla G(0,e)\rangle-\sum_{k=1}^d \nabla_k G_\h(e)\langle X, (e_k+\nabla\tilde{\phi}_k)(e)\rangle|\les \|X\|_2\frac{\log |\underline{e}|_*}{|\underline{e}|_*^d}.
\label{2scaleip}
\end{equation}
By translation invariance of the environment, we further obtain for any $x\in\Z^d$
\begin{equation}
|\langle X, \nabla G(x,e)\rangle-\sum_{k=1}^d \nabla_k G_\h(e-x)\langle X, (e_k+\nabla\tilde{\phi}_k)(e)\rangle|\les \|X\|_2\frac{\log |\underline{e}-x|_*}{|\underline{e}-x|_*^d}.
\label{2scalew}
\end{equation}

\begin{rem}
\eqref{2scaleip} is an improvement of \cite[Theorem 5.1]{mourrat2014correlation}.
\end{rem}

To prove Proposition~\ref{p:2scaleGr}, we introduce the \emph{flux corrector}, following \cite{gloria2014regularity}. For every $i=1,\ldots,d$, define $q_i=a(e_i+D\phi_i)-a_\h e_i$, which describes the current correction, and $q_{ij}$ to be its $j-$th component. By the corrector equation \eqref{cor}, we have $D^*q_i=0$, and by the expression of $a_\h$ in \eqref{effcst}, $\E\{q_{ij}\}=0$.

We need the following integrability property.

\begin{lem}
Fix any $i,j,k=1,\ldots,d$. For $\lambda>0$, let $\sigma_\lambda$ solve 
\begin{equation*}
(\lambda-\Delta)\sigma_\lambda=D_kq_{ij},
\end{equation*}
then $\sigma_\lambda$ is bounded in $L^4(\Omega)$ uniformly in $\lambda,i,j,k$. Furthermore, $\sigma_\lambda$ converges in $L^2(\Omega)$ with the limit $\sigma\in L^4(\Omega)$ and solving
\begin{equation*}
-\Delta \sigma=D_kq_{ij}.
\end{equation*}
\label{l:L4bd}
\end{lem}

\begin{proof}
We first apply the spectral gap inequality in the form given by Lemma~\ref{l:SG} to $\sigma_\lambda$ and obtain
\begin{equation*}
\E\{\sigma_\lambda^4\}\les\left(\sum_{e\in\B} \sqrt{\E\{|\sigma_\lambda-\sigma_\lambda^e|^4\}}\right)^2.
\end{equation*}
Then we compute $\sigma_\lambda-\sigma_\lambda^e$ for fixed $e\in \mathbb{B}$. Let $G_{\Delta,\lambda}$ be the Green function of $\lambda-\Delta$, we have 
\begin{eqnarray*}
\sigma_\lambda=\tilde{\sigma}_\lambda(0)=\sum_{y\in\Z^d} G_{\Delta,\lambda}(0,y)\nabla_k \tilde{q}_{ij}(y),\\
\sigma_\lambda^e=\tilde{\sigma}_\lambda^e(0)=\sum_{y\in\Z^d} G_{\Delta,\lambda}(0,y)\nabla_k \tilde{q}_{ij}^e(y).
\end{eqnarray*}
Since $q_{ij}$ is the $j$-th component of $a(e_i+D\phi_i)-a_\h e_i$ and $a_\h=\bar{a}I_d$, we have $q_{ij}=a_j1_{i=j}+a_jD_j\phi_i-\bar{a}1_{i=j}$, which implies
\begin{equation*}
|\tilde{q}_{ij}(y)-\tilde{q}_{ij}^e(y)|\les 1_{y=\underline{e}}(1+|\nabla_j\tilde{\phi}_i(y)|)+|\nabla_j\tilde{\phi}_i(y)-\nabla_j\tilde{\phi}_i^e(y)|.
\end{equation*}
Now we have
\begin{equation*}
\begin{aligned}
|\sigma_\lambda-\sigma_\lambda^e|\les &\sum_{y\in\Z^d} |\nabla_k^*G_{\Delta,\lambda}(0,y)|\left(1_{y=\underline{e}}(1+|\nabla_j\tilde{\phi}_i(y)|)+|\nabla_j\tilde{\phi}_i(y)-\nabla_j\tilde{\phi}_i^e(y)|\right)\\
=& |\nabla_k^* G_{\Delta,\lambda}(0,\underline{e})|(1+|\nabla_j\tilde{\phi}_i(\underline{e})|)+\sum_{y\in\Z^d} |\nabla_k^*G_{\Delta,\lambda}(0,y)||\nabla_j\tilde{\phi}_i(y)-\nabla_j\tilde{\phi}_i^e(y)|\\
:=& I_1+I_2,
\end{aligned}
\end{equation*}
so $\sqrt{\E\{|\sigma_\lambda-\sigma_\lambda^e|^4\}}\les \sqrt{\E\{|I_1|^4\}}+\sqrt{\E\{|I_2|^4\}}$. The homogeneous Green function satisfies $|\nabla_k^* G_{\Delta,\lambda}(0,x)|\les |x|_*^{1-d}$, thus $\sqrt{\E\{|I_1|^4\}}\les |\underline{e}|_*^{2-2d}$ by Proposition~\ref{p:esCo}. For $I_2$, we write
\begin{equation*}
|I_2|^4=\sum_{y_1,y_2,y_3,y_4\in\Z^d}\prod_{n=1}^4|\nabla_k^*G_{\Delta,\lambda}(0,y_n)||\nabla_j\tilde{\phi}_i(y_n)-\nabla_j\tilde{\phi}_i^e(y_n)|,
\end{equation*}
and by Lemmas~\ref{l:dge} and \ref{l:cvP} we have 
\begin{equation*}
\E\{|I_2|^4\}\les \left(\sum_{y\in\Z^d} \frac{1}{|y|_*^{d-1}}\frac{1}{|y-\underline{e}|_*^d}\right)^4 \les \frac{1}{|\underline{e}|_*^{(4d-4)-}},
\end{equation*}
so $\sqrt{\E\{|I_2|^4\}}\les 1/|\underline{e}|_*^{(2d-2)-}$. In summary, we have $$\E\{\sigma_\lambda^4\}\les \left(\sum_{e\in\B} \frac{1}{|\underline{e}|_*^{(2d-2)-}}\right)^2,$$ and since $d\geq 3$, we conclude $\E\{\sigma_\lambda^4\}\les 1$.

To show the convergence of $\sigma_\lambda$ in $L^2(\Omega)$, we only need to prove that $\langle \sigma_{\lambda_1},\sigma_{\lambda_2}\rangle$ converges as $\lambda_1,\lambda_2\to 0$. By the Green function representation, we have
\begin{equation*}
\langle \sigma_{\lambda_1},\sigma_{\lambda_2}\rangle=\sum_{y_1,y_2\in\Z^d}\nabla_k^*G_{\Delta,\lambda_1}(0,y_1)\nabla_k^*G_{\Delta,\lambda_2}(0,y_2)\E\{\tilde{q}_{ij}(y_1)\tilde{q}_{ij}(y_2)\}.
\end{equation*}
By Lemma~\ref{l:covqij}, $\Ll|\E\{\tilde{q}_{ij}(y_1)\tilde{q}_{ij}(y_2)\} \Rr|\les \frac{1}{|y_1-y_2|_*^{d-}}$. Furthermore $|\nabla_k^* G_{\Delta,\lambda}(0,y)|\les |y|_*^{1-d}$, by the dominated convergence theorem, we have
\begin{equation*}
\langle \sigma_{\lambda_1},\sigma_{\lambda_2}\rangle\to \sum_{y_1,y_2\in\Z^d}\nabla_k^*G_{\Delta,0}(0,y_1)\nabla_k^*G_{\Delta,0}(0,y_2)\E\{\tilde{q}_{ij}(y_1)\tilde{q}_{ij}(y_2)\}.
\end{equation*}
Therefore, $\sigma_{\lambda}$ converges in $L^2(\Omega)$. Its limit $\sigma$ is in $L^4(\Omega)$ by Fatou's lemma. By sending $\lambda\to 0$ in $(\lambda-\Delta)\sigma_\lambda=D_kq_{ij}$, we obtain $-\Delta \sigma=D_kq_{ij}$, and the proof is complete. 
\end{proof}

We can now define the flux corrector $\{\sigma_{ijk},i,j,k=1,\ldots,d\}$:
\begin{lem}
There exists a tensor field $\{\sigma_{ijk},i,j,k=1,\ldots,d\}$ such that 
\begin{itemize}
\item $\sigma_{ijk}=-\sigma_{ikj}$,
\item $\sigma_{ijk}\in L^4(\Omega)$,
\item $-\Delta\sigma_{ijk}=D_kq_{ij}-D_jq_{ik}$ and $\sum_{k=1}^d D_k^* \sigma_{ijk}=q_{ij}$.
\end{itemize}
\label{l:geneCor}
\end{lem}

\begin{proof}
For every $i,j,k=1,\ldots,d$ and $\lambda>0$, we consider the equation 
\begin{equation}
(\lambda-\Delta)\sigma_{ijk}^\lambda=D_kq_{ij}-D_jq_{ik}.
\label{degeco}
\end{equation}
Lemma~\ref{l:L4bd} ensures that $\E\{|\sigma_{ijk}^\lambda|^4\}\les 1$, that $\sigma_{ijk}^\lambda$ converges in $L^2(\Omega)$, and denoting the limit by $\sigma_{ijk}$, we have $\sigma_{ijk}\in L^4(\Omega)$ with $-\Delta \sigma_{ijk}=D_kq_{ij}-D_jq_{ik}$. 

The skew symmetry $\sigma_{ijk}=-\sigma_{ikj}$ is clear by \eqref{degeco}.

To show $\sum_{k=1}^d D_k^* \sigma_{ijk}=q_{ij}$, it suffices to prove $\Delta(\sum_{k=1}^d D_k^* \sigma_{ijk}-q_{ij})=0$. Indeed, $D(\sum_{k=1}^d D_k^* \sigma_{ijk}-q_{ij})=0$ implies $\sum_{k=1}^d D_k^* \sigma_{ijk}-q_{ij}=\mathrm{const}$ by ergodicity, and since $\E\{\sum_{k=1}^d D_k^* \sigma_{ijk}\}=\E\{q_{ij}\}=0$, we have $\sum_{k=1}^d D_k^* \sigma_{ijk}=q_{ij}$. Now we consider
\begin{equation*}
\begin{aligned}
\Delta(\sum_{k=1}^d D_k^* \sigma_{ijk}-q_{ij})=&\lim_{\lambda\to 0}\Delta(\sum_{k=1}^d D_k^* \sigma_{ijk}^\lambda-q_{ij})\\
=&\lim_{\lambda\to 0}\sum_{k=1}^d D_k^*(D_jq_{ik}-D_kq_{ij}+\lambda \sigma_{ijk}^\lambda)-\Delta q_{ij}.
\end{aligned}
\end{equation*}
Since $\sum_{k=1}^d D_k^* q_{ik}=0$ and $\sigma_{ijk}^\lambda$ is uniformly bounded in $L^4(\Omega)$, we have $\Delta(\sum_{k=1}^d D_k^* \sigma_{ijk}-q_{ij})=0$, and this completes the proof.
\end{proof}

\begin{proof}[Proof of Proposition~\ref{p:2scaleGr}] We follow the proof of \cite[Theorem 5.1]{mourrat2014correlation}, but use the flux corrector as in \cite{gloria2014regularity} to simplify calculations. 
Define 
\begin{equation*}
z(x):=G(0,x)-G_\h(x)-\sum_{k=1}^d \nabla_kG_\h(x)\tilde{\phi}_k(x),
\end{equation*}
as the remainder in the two-scale expansion of the Green function, the matrix function $R:=-[q_1,\ldots,q_d]$ by
\begin{equation*}
\tilde{R}_{ij}(x)=-\tilde{q}_{ji}(x)=\bar{a}1_{i=j}-\tilde{a}_i(x)(1_{i=j}+\nabla_i\tilde{\phi}_j(x)),
\end{equation*}
and $\tilde{h}:\Z^d\to\R$ by
\begin{equation*}
\tilde{h}(x)=-\sum_{i=1}^d \nabla_i^*\left(\tilde{a}_i(x)\sum_{j=1}^d \tilde{\phi}_j(x+e_i)\nabla_i\nabla_jG_\h(x)\right).
\end{equation*}
By \cite[Proposition 5.6]{mourrat2014correlation}, we have
\begin{equation}
z(x)=\sum_{y\in\Z^d} G(x,y)\sum_{i,j=1}^d \tilde{R}_{ij}(y-e_i)\nabla_i^*\nabla_jG_\h(y)+\sum_{y\in\Z^d} G(x,y)\tilde{h}(y).
\label{eqz}
\end{equation}
Consider the first term on the right-hand side of \eqref{eqz}. Since $R_{ij}=-q_{ji}=-\sum_{k=1}^d D_k^* \sigma_{jik}$ by Lemma~\ref{l:geneCor}, we can write
\begin{equation*}
\begin{aligned}
\sum_{i,j=1}^d \tilde{R}_{ij}(y-e_i)\nabla_i^*\nabla_jG_\h(y)=&-\sum_{i,j,k=1}^d \nabla_k^*\tilde{\sigma}_{jik}(y-e_i)\nabla_i^*\nabla_jG_\h(y)\\
=&-\sum_{i,j,k=1}^d \nabla_k^*(\tilde{\sigma}_{jik}(y-e_i)\nabla_i^*\nabla_jG_\h(y)),
\end{aligned}
\end{equation*}
where the last equality uses the fact $\nabla_k^*(f(x)g(x))=\nabla_k^*f(x) g(x)+f(x-e_k)\nabla_k^*g(x)$ and $\sigma_{jik}+\sigma_{jki}=0$. Therefore, by using the flux corrector $\sigma$, we can write $\sum_{i,j=1}^d \tilde{R}_{ij}(y-e_i)\nabla_i^*\nabla_jG_\h(y)$ in divergence form. Note that $\tilde{h}$ is in divergence form. An integration by parts leads to
\begin{equation*}
\begin{aligned}
z(x)=&-\sum_{y\in\Z^d}\sum_{k=1}^d\nabla_{y,k} G(x,y)\sum_{i,j=1}^d \tilde{\sigma}_{jik}(y-e_i)\nabla_i^*\nabla_jG_\h(y)\\
&-\sum_{y\in\Z^d} \sum_{i=1}^d \nabla_{y,i}G(x,y)\tilde{a}_i(y)\sum_{j=1}^d \tilde{\phi}_j(y+e_i)\nabla_i\nabla_jG_\h(y),
\end{aligned}
\end{equation*}
so for $e\in \mathbb{B}$, we have
\begin{equation*}
\begin{aligned}
\nabla z(e)=&-\sum_{y\in\Z^d} \sum_{k=1}^d\nabla \nabla_{y,k} G(e,y)\sum_{i,j=1}^d \tilde{\sigma}_{jik}(y-e_i)\nabla_i^*\nabla_jG_\h(y)\\
&-\sum_{y\in\Z^d} \sum_{i=1}^d \nabla\nabla_{y,i}G(e,y)\tilde{a}_i(y)\sum_{j=1}^d \tilde{\phi}_j(y+e_i)\nabla_i\nabla_jG_\h(y).
\end{aligned}
\end{equation*}

Note that
\begin{equation*}
\begin{aligned}
\nabla_iz(x)=&\nabla_iG(0,x)-\nabla_iG_\h(x)-\sum_{k=1}^d (\nabla_i\tilde{\phi}_k(x)\nabla_kG_\h(x)+\tilde{\phi}_k(x+e_i)\nabla_i\nabla_kG_\h(x))\\
=&\left(\nabla_iG(0,x)-\nabla_iG_\h(x)-\sum_{k=1}^d \nabla_i\tilde{\phi}_k(x)\nabla_kG_\h(x)\right)-\sum_{k=1}^d\tilde{\phi}_k(x+e_i)\nabla_i\nabla_kG_\h(x),
\end{aligned}
\end{equation*}
and moreover, by the moments bounds on $\phi_k$ provided by Proposition~\ref{p:esCo} and the fact that $|\nabla_i\nabla_jG_\h(x)|\les |x|_*^{-d}$, we have
\begin{equation*}
\|\tilde{\phi}_k(x+e_i)\nabla_i\nabla_kG_\h(x)\|_2\les \frac{1}{|x|_*^d}\les  \frac{\log |x|_*}{|x|_*^d}.
\end{equation*}
As a consequence, in order to prove Proposition~\ref{p:2scaleGr}, it is enough to show that
\begin{equation}
\| \nabla z(e)\|_2\les \frac{\log |\underline{e}|_*}{|\underline{e}|_*^d}.
\label{2sgrb}
\end{equation}
%

In order to prove \eqref{2sgrb}, we note that $\nabla z(e)$ is a finite linear combination of terms in the form $\sum_{y\in\Z^d} \nabla\nabla_{y,k}G(e,y)f(y)\nabla_i^*\nabla_j G_\h(y)$ or 
$\sum_{y\in\Z^d} \nabla\nabla_{y,k}G(e,y)f(y)\nabla_i\nabla_j G_\h(y)$ for some $i,j,k$ and $f$. Clearly, they can be bounded by $\sum_{y\in\Z^d} |\nabla\nabla_{y,k}G(e,y)f(y)||y|_*^{-d}$, so we have
\begin{equation*}
\begin{aligned}
\| \nabla z(e)\|_2\les &\|\sum_{y\in\Z^d} |\nabla\nabla_{y,k}G(e,y)f(y)||y|_*^{-d}\|_2\\
\leq & \sum_{y\in\Z^d} |y|_*^{-d}\|\nabla\nabla_{y,k}G(e,y)f(y)\|_2\\
\leq &\sum_{y\in\Z^d} |y|_*^{-d}\|\nabla\nabla_{y,k}G(e,y)\|_4\|f(y)\|_4.
\end{aligned}
\end{equation*}
When $f=\tilde{\sigma}_{jik}$ or $\tilde{a}_i\tilde{\phi}_j$, $\|f(y)\|_4$ is uniformly bounded by Lemma~\ref{l:geneCor} and  Proposition~\ref{p:esCo}, thus by applying Proposition~\ref{p:esGr} and Lemma~\ref{l:cvP}, we obtain
\begin{equation*}
\|\nabla z(e)\|_2\les \sum_{y\in\Z^d} \frac{1}{|y|_*^{d}}\frac{1}{|y-\underline{e}|_*^{d}}\les \frac{\log |\underline{e}|_*}{|\underline{e}|_*^d}.
\end{equation*}
The proof of Proposition \ref{p:2scaleGr} is complete.
\end{proof}

\subsection{Second-order Poincar\'e inequality}

Let $d_K$ be the Kantorovich-Wasserstein distance
\begin{equation*}
d_K(X,Y)=\sup\{\E\{h(X)\}-\E\{h(Y)\}: \|h'\|_\infty\leq 1\}.
\end{equation*}
In order to show that the rescaled fluctuations are asymptotically Gaussian, we will use the following result.
\begin{prop}\cite[Proposition~2.1]{MN}
\label{p:2poin}
Let $F\in L^2(\Omega)$ be such that $\E\{F\}=0$ and $\E\{F^2\}=1$. Assume also that $F$ has weak derivatives satisfying $\sum_e \E\{|\partial_e F|^4\}^\frac12<\infty$ and $\E\{|\partial_e\partial_{e'}F|^4\}<\infty$ for all $e,e'\in \B$. Let $Y\sim N(0,1)$. Then
\begin{equation}
d_K(F,Y)\leq \sqrt{\frac{5}{\pi}}\sqrt{\sum_{e'\in\B}\left(\sum_{e\in\B}\|\partial_e F\|_4\|\partial_e\partial_{e'} F\|_4\right)^2}.
\end{equation}
\end{prop}

Using the above result, we only need to show the following lemma to prove \eqref{conN}.
\begin{lem}
Let
\begin{equation}
\kappa^2:=\sum_{e'\in\B}\left(\|\partial_e\U_\eps^{(f)}(g)\|_4\sum_{e\in\B}\|\partial_{e}\partial_{e'}\U_\eps^{(f)}(g)\|_4\right)^2.
\label{bdk3}
\end{equation}
If $\sigma_g^2$ defined in \eqref{mainVar} is not zero, then
\begin{equation*}
\frac{\kappa^2}{\var^2\{\U_\eps^{(f)}(g)\}}\les \eps^d|\log\eps|^2.
\end{equation*}
\label{l:k0k3}
\end{lem}

\begin{rem}
\label{r:rate}
By Proposition~\ref{p:2poin} and Lemma~\ref{l:k0k3}, if $\sigma_g^2 \neq 0$, then we actually obtain the convergence rate
\begin{equation*}
d_K\left(\frac{\U_\eps^{(f)}(g)-\E\{\U_\eps^{(f)}(g)\}}{\sqrt{\var\{\U_\eps^{(f)}(g)\}}},N(0,1)\right)\les \eps^{\frac{d}{2}}|\log\eps|.
\end{equation*}
\end{rem}

\section{Convergence of the variance}
\label{s:cv}
The aim of this section is to prove \eqref{conV}. 

Recall that $\U_\eps^{(f)}(g)=\eps^{d+2}\sum_{x,y\in \Z^d}G(x,y)g(\eps x)f(\eps y)$, so
\begin{equation*}
\begin{aligned}
\var\{\U_\eps^{(f)}(g)\}=&\eps^{2d+4}\sum_{x,y,z,w\in \Z^d} \mathrm{Cov}\{G(x,y),G(z,w)\}g(\eps x)f(\eps y)g(\eps z)f(\eps w)\\
=&\eps^{2d+4}\sum_{x,y,z,w\in \eps \Z^d} \mathrm{Cov}\{G(\frac{x}{\eps},\frac{y}{\eps}),G(\frac{z}{\eps},\frac{w}{\eps})\}g(x)f(y)g(z)f(w).
\end{aligned}
\end{equation*}
The covariance is given explicitly by the Helffer-Sj\"ostrand representation
\begin{equation}
\mathrm{Cov}\{G(x,y),G(z,w)\}=\sum_{e\in\B} \langle \partial_e G(x,y), (1+\L)^{-1} \partial_e G(z,w)\rangle,
\label{covGr}
\end{equation}
and since $\partial_e G(x,y)=-\partial_e a_e\nabla G(x,e)\nabla G(y,e)$ by Lemma~\ref{l:vd}, \eqref{covGr} is rewritten as
\begin{equation}
\begin{aligned}
&\mathrm{Cov}\{G(x,y),G(z,w)\}\\
=&\sum_{e\in\B} \langle \partial_ea_e \nabla G(x,e)\nabla G(y,e), (1+\L)^{-1} \partial_ea_e \nabla G(z,e)\nabla G(w,e)\rangle.
\label{covGr1}
\end{aligned}
\end{equation}
To prove the convergence of $\eps^{-d}\var\{\U_\eps^{(f)}(g)\}$, we use the two-scale expansion of the Green function obtained in Proposition~\ref{p:2scaleGr}.
For $e\in \mathbb{B}$, $x_i\in \Z^d, i=1,2,3,4$, define 
\begin{equation}
\mathcal{E}(x_1,x_2,x_3,x_4):=\sum_{e\in\B}\sum_{i=1}^4\frac{\log |\underline{e}-x_i|_*}{|\underline{e}-x_i|_*^d}\prod_{j=1, j\neq i}^4
\frac{1}{|\underline{e}-x_j|_*^{d-1}}
\label{ercoEx}
\end{equation}
and 
\begin{equation*}
\mathcal{K}(x,y,z,w):=\sum_{i,j,k,l=1}^d \tilde{K}_{ijkl}\sum_{v\in \Z^d}\nabla_iG_\h(v-x)\nabla_jG_\h(v-y)\nabla_kG_\h(v-z)\nabla_lG_\h(v-w)
\end{equation*}
with $\tilde{K}_{ijkl}$ given by \eqref{excos}.

\begin{prop}
$|\mathrm{Cov}\{G(x,y),G(z,w)\}-\mathcal{K}(x,y,z,w)|\les \mathcal{E}(x,y,z,w)$.
\label{p:exCov}
\end{prop}

\begin{proof}
Each term in \eqref{covGr1} contains four factors of gradient of the Green function. We first consider $\nabla G(x,e)$ and let
\begin{equation*}
X=\partial_ea_e \nabla G(y,e)(1+\L)^{-1} \partial_ea_e \nabla G(z,e)\nabla G(w,e).
\end{equation*} 
By \eqref{2scalew}, we have
\begin{equation*}
|\langle X, \nabla G(x,e)\rangle|-\sum_{k=1}^d \nabla_k G_\h(e-x)\langle X, (e_k+\nabla\tilde{\phi}_k)(e)\rangle|\les \|X\|_2\frac{\log |\underline{e}-x|_*}{|\underline{e}-x|_*^d}.
\end{equation*}
By Proposition~\ref{p:esGr} and the fact that $(1+\L)^{-1}$ is a contraction from $L^p(\Omega)$ to $L^p(\Omega)$ for any $p\geq 2$, we have $\|X\|_2\les |\underline{e}-y|_*^{1-d}|\underline{e}-z|_*^{1-d}|\underline{e}-w|_*^{1-d}$, so
\begin{equation*}
|\mathrm{Cov}\{G(x,y),G(z,w)\}-\sum_{e\in\B}\sum_{k=1}^d \nabla_k G_\h(e-x)\langle X, (e_k+\nabla\tilde{\phi}_k)(e)\rangle|\les \mathcal{E}(x,y,z,w).
\end{equation*}
Now we carry out the same argument for $\nabla G(y,e),\nabla G(z,e),\nabla G(w,e)$, and in the end obtain
\begin{equation*}
|\mathrm{Cov}\{G(x,y),G(z,w)\}-\mathcal{K}(x,y,z,w)|\les \mathcal{E}(x,y,z,w).
\end{equation*}
The proof is complete.
\end{proof}

Proposition~\ref{p:exCov} leads to
\begin{equation*}
\begin{aligned}
&|\eps^{-d}\var\{\U_\eps^{(f)}(g)\}-\eps^{d+4}\sum_{x,y,z,w\in \eps\Z^d} \mathcal{K}(\frac{x}{\eps},\frac{y}{\eps},\frac{z}{\eps},\frac{w}{\eps})g(x)f(y)g(z)f(w)|\\
\les & \eps^{d+4}\sum_{x,y,z,w\in \eps\Z^d}\mathcal{E}(\frac{x}{\eps},\frac{y}{\eps},\frac{z}{\eps},\frac{w}{\eps})|g(x)f(y)g(z)f(w)|.
\end{aligned}
\end{equation*}
Hence, the proof of \eqref{conV} will be complete once we have proved the following two lemmas.

\begin{lem}
$\eps^{d+4}\sum_{x,y,z,w\in \eps\Z^d}\mathcal{E}(\frac{x}{\eps},\frac{y}{\eps},\frac{z}{\eps},\frac{w}{\eps})|g(x)f(y)g(z)f(w)|\to 0$ as $\eps\to 0$.
\label{l:conVre}
\end{lem}

\begin{lem}
$\eps^{d+4}\sum_{x,y,z,w\in \eps\Z^d} \mathcal{K}(\frac{x}{\eps},\frac{y}{\eps},\frac{z}{\eps},\frac{w}{\eps})g(x)f(y)g(z)f(w)\to \sigma_g^2$ as $\eps\to 0$.
\label{l:conVma}
\end{lem}

In the following, we assume $|g|,|f|\leq h$ for some $h\in\C_c(\R^d)$.

\begin{proof}[Proof of Lemma~\ref{l:conVre}]
By denoting the number of different elements in $\{x,y,z,w\}$ by $s$, we decompose $\sum_{x,y,z,w\in \eps\Z^d}=\sum1_{s=1}+\sum1_{s=2}+\sum1_{s=3}+\sum1_{s=4}$. The following estimates are obtained with an application of Lemma~\ref{l:erCov}.

When $s=1$,
\begin{equation*}
\eps^{d+4}\sum1_{s=1}\mathcal{E}(\frac{x}{\eps},\frac{y}{\eps},\frac{z}{\eps},\frac{w}{\eps})|g(x)f(y)g(z)f(w)|\les \eps^{d+4}\sum_{x\in\eps\Z^d}h(x)^4.
\end{equation*}
Since $h\in \C_c(\R^d)$, $\eps^d\sum_{x\in \eps \Z^d} h(x)^4\to \int_{\R^d}h(x)^4dx$ which is bounded, so we have the r.h.s. of the above display goes to zero as $\eps\to 0$.

When $s=2$, 
\begin{equation*}
\begin{aligned}
&\eps^{d+4}\sum1_{s=2}\mathcal{E}(\frac{x}{\eps},\frac{y}{\eps},\frac{z}{\eps},\frac{w}{\eps})|g(x)f(y)g(z)f(w)| \\
\les &\eps^{d+4}\sum_{x \neq y \in \eps \Z^d} h(x)^2h(y)^2|\frac{x}{\eps}-\frac{y}{\eps}|_*^{2-2d}+h(x)^3h(y)|\frac{x}{\eps}-\frac{y}{\eps}|_*^{1-d}.
\end{aligned}
\end{equation*}
For $x\neq 0 \in \Z^d$, we have $|x|_*> |x| \geq 1$, so the r.h.s. of the above display is bounded by
\[
\begin{aligned}
&\eps^{d+4}\sum_{x \neq y \in \eps \Z^d} (h(x)^2h(y)^2+h(x)^3h(y))|\frac{x}{\eps}-\frac{y}{\eps}|_*^{1-d}\\
\leq & \eps^{2d+3}\sum_{x \neq y \in \eps \Z^d} (h(x)^2h(y)^2+h(x)^3h(y))|x-y|^{1-d}.
\end{aligned}
\]
Similarly, $\eps^{2d}\sum_{x \neq y \in \eps \Z^d} (h(x)^2h(y)^2+h(x)^3h(y))|x-y|^{1-d}$ converges as a Riemann sum, which implies
\[
\eps^{2d+3}\sum_{x \neq y \in \eps \Z^d} (h(x)^2h(y)^2+h(x)^3h(y))|x-y|^{1-d}\sim \eps^3\to 0
\]
as $\eps\to 0$.

The discussion for $s=3,4$ is similar to $s=2$, so we omit the details.

When $s=3$,
\begin{equation*}
\begin{aligned}
&\eps^{d+4}\sum1_{s=3}\mathcal{E}(\frac{x}{\eps},\frac{y}{\eps},\frac{z}{\eps},\frac{w}{\eps})|g(x)f(y)g(z)f(w)|\\
\les & \eps^{3d+2}\sum_{\substack{x,y,z \in \eps \Z^d \\|\{x,y,z\}| = 3}}h(x)^2h(y)h(z)\frac{1}{|x-y|^{d-1}}\left(\frac{1}{|x-z|^{d-1}}+\frac{1}{|y-z|^{d-1}}\right)\sim \eps^2.
\end{aligned}
\end{equation*}

When $s=4$,
\begin{equation*}
\begin{aligned}
&\eps^{d+4}\sum1_{s=4}\mathcal{E}(\frac{x}{\eps},\frac{y}{\eps},\frac{z}{\eps},\frac{w}{\eps})|g(x)f(y)g(z)f(w)|\\
\les & \eps^{(4d+1)-}\sum1_{s=4}h(x)h(y)h(z)h(w) \frac{1}{|x-y|^{(d-1)-}}\frac{1}{|x-z|^{(d-1)-}}\frac{1}{|x-w|^{(d-1)-}}\sim \eps^{1-}.
\end{aligned}
\end{equation*}
The proof is complete.
\end{proof}

\begin{proof}[Proof of Lemma~\ref{l:conVma}]
Recall that 
\begin{equation*}
\mathcal{K}(x,y,z,w)=\sum_{i,j,k,l=1}^d\tilde{K}_{ijkl}\sum_{v\in \Z^d}\nabla_iG_\h(v-x)\nabla_jG_\h(v-y)\nabla_kG_\h(v-z)\nabla_lG_\h(v-w).
\end{equation*}
By defining $\mathcal{F}_{ijkl}(v,x,y,z,w):=\nabla_iG_\h(v-x)\nabla_jG_\h(v-y)\nabla_kG_\h(v-z)\nabla_lG_\h(v-w)$, we only need to show the convergence of 
\begin{equation*}
I_{ijkl}=\eps^{d+4}\sum_{x,y,z,w\in \eps\Z^d}\sum_{v\in \Z^d}\mathcal{F}_{ijkl}(v,\frac{x}{\eps},\frac{y}{\eps},\frac{z}{\eps},\frac{w}{\eps})g(x)f(y)g(z)f(w)
\end{equation*}
for fixed $i,j,k,l$.

We claim that $\nabla_iG_\h(v-\frac{x}{\eps})$ can be replaced by $\partial_{x_i}\cG_\h(v-\frac{x}{\eps})$ in $\mathcal{F}_{ijkl}(v,\frac{x}{\eps},\frac{y}{\eps},\frac{z}{\eps},\frac{w}{\eps})$ of the above expression with the sum over $v\neq x/\eps$. Indeed, by \cite[Proposition~A.3]{mourrat2014correlation}, for $x\neq 0$, 
\begin{equation*}
|\nabla_iG_\h(x)-\partial_{x_i} \cG_\h(x)|\les |x|^{-d}.
\end{equation*}
If we define $\mathcal{F}_{jkl}^i(v,x,y,z,w):=\partial_{x_i}\cG_\h(v-x)\nabla_jG_\h(v-y)\nabla_kG_\h(v-z)\nabla_lG_\h(v-w)$, the error induced by the replacement can be estimated as
\begin{equation*}
|I_{ijkl}-\eps^{d+4}\sum_{x,y,z,w\in \eps\Z^d}\sum_{v\neq \frac{x}{\eps}}\mathcal{F}_{jkl}^i(v,\frac{x}{\eps},\frac{y}{\eps},\frac{z}{\eps},\frac{w}{\eps})g(x)f(y)g(z)f(w)|\les J_1+J_2,
\end{equation*}
with
\begin{equation*}
J_1=\eps^{d+4}\sum_{x,y,z,w\in \eps\Z^d}|\mathcal{F}_{ijkl}(\frac{x}{\eps},\frac{x}{\eps},\frac{y}{\eps},\frac{z}{\eps},\frac{w}{\eps})|h(x)h(y)h(z)h(w).
\end{equation*}
and 
\begin{equation*}
J_2=\eps^{d+4}\sum_{x,y,z,w\in \eps\Z^d}\sum_{v\neq \frac{x}{\eps}}|v-\frac{x}{\eps}|^{-d}|v-\frac{y}{\eps}|_*^{1-d}|v-\frac{z}{\eps}|_*^{1-d}|v-\frac{w}{\eps}|_*^{1-d}h(x)h(y)h(z)h(w).
\end{equation*}
For $J_1$, by using $|\nabla_i G_\h(x)|\les |x|_*^{1-d}$ and considering different cases according to whether $y,z,w=x$ as in the proof of Lemma~\ref{l:conVre}, we obtain
\begin{equation*}
J_1\les \eps^{d+4}\sum_{x,y,z,w\in \eps\Z^d}|\frac{x-y}{\eps}|_*^{1-d}|\frac{x-z}{\eps}|_*^{1-d}|\frac{x-w}{\eps}|_*^{1-d}h(x)h(y)h(z)h(w)\to 0.
\end{equation*}
For $J_2$, we note that $\sum_{v\neq \frac{x}{\eps}}|v-\frac{x}{\eps}|^{-d}|v-\frac{y}{\eps}|_*^{1-d}|v-\frac{z}{\eps}|_*^{1-d}|v-\frac{w}{\eps}|_*^{1-d}\les \mathcal{E}(\frac{x}{\eps},\frac{y}{\eps},\frac{z}{\eps},\frac{w}{\eps})$ with $\mathcal{E}(x,y,z,w)$ defined in \eqref{ercoEx}, so we can apply Lemma~\ref{l:conVre} to show $J_2\to 0$. The claim is proved.

By following the same argument for $\nabla_jG_\h(v-\frac{y}{\eps}), \nabla_k G_\h(v-\frac{z}{\eps}), \nabla_l G_\h(v-\frac{w}{\eps})$, we derive
\begin{equation*}
|I_{ijkl}-\eps^{d+4}\sum_{x,y,z,w\in \eps\Z^d} \sum_{v\neq \frac{x}{\eps},\frac{y}{\eps},\frac{z}{\eps},\frac{w}{\eps}} \mathcal{F}^{ijkl}(v,\frac{x}{\eps},\frac{y}{\eps},\frac{z}{\eps},\frac{w}{\eps}) g(x)f(y)g(z)f(w)|\to 0,
\end{equation*}
with $\mathcal{F}^{ijkl}(v,x,y,z,w):=\partial_{x_i}\cG_\h(v-x)\partial_{x_j}\cG_\h(v-y)\partial_{x_k}\cG_\h(v-z)\partial_{x_l}\cG_\h(v-w)$. Since $d\geq 3$, $\cG_\h(x)=c_\h|x|^{2-d}$ for some constant $c_\h$, and $\partial_{x_i}\cG_\h(x)=c_\h(2-d)x_i/|x|^d$, so we have
\begin{equation*}
\begin{aligned}
&\eps^{d+4}\sum_{x,y,z,w\in \eps\Z^d} \sum_{v\neq \frac{x}{\eps},\frac{y}{\eps},\frac{z}{\eps},\frac{w}{\eps}} \mathcal{F}^{ijkl}(v,\frac{x}{\eps},\frac{y}{\eps},\frac{z}{\eps},\frac{w}{\eps}) g(x)f(y)g(z)f(w)\\
=&\eps^{5d}\sum_{x,y,z,w,v\in \eps\Z^d} 1_{v\neq x,y,z,w} 
\mathcal{F}^{ijkl}(v,x,y,z,w)g(x)f(y)g(z)f(w)\\
\to &\int_{\R^{5d}}\mathcal{F}^{ijkl}(v,x,y,z,w)g(x)f(y)g(z)f(w)dxdydzdwdv.
\end{aligned}
\end{equation*}
The proof is complete.
\end{proof}

\section{Convergence to a Gaussian when \texorpdfstring{$\sigma_g^2>0$}{sigma > 0}}
\label{s:cg}

Recall that in order to prove \eqref{conN}, that is,
\begin{equation*}
\frac{\U_\eps^{(f)}(g)-\E\{\U_\eps^{(f)}(g)\}}{\sqrt{\var\{\U_\eps^{(f)}(g)\}}}\Rightarrow N(0,1),
\end{equation*}
we only need to show Lemma~\ref{l:k0k3}. 
\begin{proof}[Proof of Lemma~\ref{l:k0k3}]
We first prepare the ground by estimating the terms appearing in the definition of $\kappa^2$. 
By a direct calculation, we have
\begin{equation*}
\partial_e\U_\eps^{(f)}(g)=-\eps^{d+2}\sum_{x,y\in \Z^d}\partial_e a_e \nabla G(x,e)\nabla G(y,e) f(\eps y)g(\eps x),
\end{equation*}
and
\begin{equation*}
\begin{aligned}
\partial_{e'}\partial_e \U_\eps^{(f)}(g)=&-\eps^{d+2}\sum_{x,y\in\Z^d}\partial_e^2a_e\nabla G(x,e)\nabla G(y,e) f(\eps y)g(\eps x)1_{e'=e}\\
&+\eps^{d+2}\sum_{x,y\in\Z^d}\partial_ea_e\partial_{e'}a_{e'}\nabla G(x,e')\nabla\nabla G(e,e')\nabla G(y,e) f(\eps y)g(\eps x)\\
&+\eps^{d+2}\sum_{x,y\in\Z^d}\partial_ea_e\partial_{e'}a_{e'}\nabla G(x,e)\nabla G(y,e')\nabla\nabla G(e,e')f(\eps y)g(\eps x).
\end{aligned}
\end{equation*}
By Proposition~\ref{p:esGr} and the fact that $a_e=\eta(\zeta_e)$ with $|\eta'|,|\eta''|$ uniformly bounded, we have for any $p\geq 1$ (with the multiplicative constant depending on $p$):
\begin{equation*}
\|\partial_e \U_\eps^{(f)}(g)\|_p\les \eps^{d+2}\sum_{x,y\in\Z^d}\frac{1}{|x-\underline{e}|_*^{d-1}}\frac{1}{|y-\underline{e}|_*^{d-1}}|f(\eps y)g(\eps x)|,
\end{equation*}
and
\begin{equation*}
\begin{aligned}
\|\partial_{e'}\partial_e \U_\eps^{(f)}(g)\|_p\les &\eps^{d+2}\sum_{x,y\in\Z^d}\frac{1}{|x-\underline{e}|_*^{d-1}}\frac{1}{|y-\underline{e}|_*^{d-1}}|f(\eps y)g(\eps x)|1_{e'=e}\\
&+\eps^{d+2}\sum_{x,y\in\Z^d}\frac{1}{|x-\underline{e'}|_*^{d-1}}\frac{1}{|\underline{e}-\underline{e'}|_*^d}\frac{1}{|y-\underline{e}|_*^{d-1}}|f(\eps y)g(\eps x)|\\
&+\eps^{d+2}\sum_{x,y\in\Z^d}\frac{1}{|x-\underline{e}|_*^{d-1}}\frac{1}{|\underline{e}-\underline{e'}|_*^d}\frac{1}{|y-\underline{e'}|_*^{d-1}}|f(\eps y)g(\eps x)|.\end{aligned}
\end{equation*}
Since $f,g$ are both bounded and compactly supported, we apply Lemma~\ref{l:MN} to obtain
\begin{equation}
\|\partial_e \U_\eps^{(f)}(g)\|_p\les \eps^{d+2} \left(\sum_{x\in\Z^d}\frac{1}{|x-\underline{e}|_*^{d-1}}1_{|x|\les \eps^{-1}}\right)^2\les
\frac{\eps^{d}}{|\eps \underline{e}|_*^{2d-2}},
\label{peUeps}
\end{equation}
and
\begin{equation}
\begin{aligned}
\|\partial_{e'}\partial_e \U_\eps^{(f)}(g)\|_p\les &\eps^{d+2} \left(\sum_{x\in\Z^d}\frac{1}{|x-\underline{e}|_*^{d-1}}1_{|x|\les \eps^{-1}}\right)^2(1_{e'=e}+\frac{1}{|\underline{e}-\underline{e'}|_*^d})\\
\les &\frac{\eps^{d}}{|\eps \underline{e}|_*^{2d-2}} \ \frac{1}{|\underline{e}-\underline{e'}|_*^d}.
\end{aligned}
\label{ppeUeps}
\end{equation}
We are now ready to estimate $\kappa^2$. By \eqref{peUeps} and \eqref{ppeUeps}, we have
\begin{equation*}
\begin{aligned}
\kappa^2&=\sum_{e'\in\B}\left(\sum_{e\in\B}\|\partial_{e'}\partial_e\U_\eps^{(f)}(g)\|_4\|\partial_e\U_\eps^{(f)}(g)\|_4\right)^2\\
&\les\sum_{e'\in\B}\left(\sum_{e\in\B}\frac{\eps^{d}}{|\eps \underline{e}|_*^{2d-2}}\ \frac{1}{|\underline{e}-\underline{e'}|_*^d} \ \frac{\eps^{d}}{|\eps \underline{e}|_*^{2d-2}}\right)^2.
\end{aligned}
\end{equation*}
Applying Lemma~\ref{l:MN}, we get
\begin{equation*}
\begin{aligned}
\sum_{e'\in\B}\left(\sum_{e\in\B}\frac{\eps^{2d}}{|\eps\underline{e}|_*^{4d-4}}\frac{1}{|\underline{e}-\underline{e'}|_*^d}\right)^2\les& \sum_{e'\in\B}\eps^{4d}|\log\eps|^2\frac{1}{|\eps \underline{e'}|_*^{2d}}\\
\les &\eps^{4d}|\log\eps|^2\sum_{x\in\Z^d}\frac{1}{(2+|\eps x|)^{2d}}\\
\les &\eps^{3d}|\log\eps|^2.
\end{aligned}
\end{equation*}
To sum up, $\kappa^2\les \eps^{3d}|\log\eps|^2$. By \eqref{conV}, $\var\{\U_\eps^{(f)}(g)\}\gtrsim \eps^d$ if $\sigma_g^2\neq 0$, which leads to
\begin{equation*}
\frac{\kappa^2}{\var^2\{\U_\eps^{(f)}(g)\}}\les \frac{\eps^{3d}|\log\eps|^2}{\eps^{2d}}\les \eps^d|\log\eps|^2\to 0,
\end{equation*}
and the proof is complete.
\end{proof}

\section{Tightness in \texorpdfstring{$\C_{\loc}^\alpha$}{C alpha}}
\label{s:tt}

Roughly speaking, for $\alpha <0$, a distribution $F$ is $\alpha$-Hölder regular around the point $x \in \R^d$ if for every smooth, compactly supported test function $\chi$, we have
\begin{equation}
\label{e:behav}
  F[\eps^{-d} \,\chi(\eps^{-1} ( \, \cdot \, - x))] \lesssim \eps^{-\alpha} \qquad (\eps \to 0).
\end{equation}
By \cite[Theorem~2.25]{mourrat2015tight}, in order to prove that $\mathscr{U}_\eps^{(f)}$ is tight in $\C_{\loc}^\alpha$ for every $\alpha<1-\frac{d}{2}$, we only need to prove the following proposition.

\begin{prop}
For any $g\in \C_c(\R^d)$, let $g_\lambda(x)=\lambda^{-d}g(x/\lambda)$. For all $p\geq 1$, there exists a constant $C=C(p,g)$ such that 
\begin{equation*}
\|\mathscr{U}_\eps^{(f)}(g_\lambda)\|_p\leq C\lambda^{1-\frac{d}{2}},
\end{equation*}
\end{prop}

\begin{proof}
We follow the proof of \cite[Proposition~3.1]{MN}. 
Recall that $\mathscr{U}_\eps^{(f)}(g)=\eps^{-\frac{d}{2}}(\U_\eps^{(f)}(g)-\E\{\U_\eps^{(f)}(g)\})$, by writing
\begin{equation*}
\mathscr{U}_\eps^{(f)}(g_\lambda)\lambda^{\frac{d}{2}-1}=\frac{\eps^{\frac{d}{2}+2}}{\lambda^{\frac{d}{2}+1}}\sum_{x\in \Z^d}\sum_{y\in\Z^d}(G(x,y)-\E\{G(x,y)\})f(\eps y)g(\frac{\eps x}{\lambda})=:X_{\eps,\lambda},
\end{equation*}
and we need to show that $X_{\eps,\lambda}$ is uniformly bounded in $L^p(\Omega)$. Since $\E\{X_{\eps,\lambda}\}=0$, in particular
\begin{equation}
\Ll(\E\{X_{\eps,\lambda}^p\}\Rr)^2\les 1
\label{pftight}
\end{equation}
 holds for $p=1$. We argue inductively, assuming that \eqref{pftight} holds for some $p=n$, and showing that it also holds for $p=2n$, which would complete the proof.

Since $\E\{X_{\eps,\lambda}^{2n}\}=\E\{X_{\eps,\lambda}^n\}^2+\var\{X_{\eps,\lambda}^n\}\les 1+\var\{X_{\eps,\lambda}^n\}$, it suffices to show $\var\{X_{\eps,\lambda}^n\}\les \E\{X_{\eps,\lambda}^{2n}\}^{1-\frac{1}{n}}$. By the spectral gap inequality (see \cite[Corollary 3.3]{mourrat2014correlation}), we have
\begin{equation*}
\begin{aligned}
\var\{X_{\eps,\lambda}^n\}\leq &\sum_{e\in\B}\E\{|\partial_e X_{\eps,\lambda}^n|^2\}=\sum_{e\in\B}\E\{|nX_{\eps,\lambda}^{n-1}\partial_e X_{\eps,\lambda}|^2\}\\
\les & \E\{|X_{\eps,\lambda}|^{2n}\}^{1-\frac{1}{n}}\sum_{e\in\B} \E\{|\partial_e X_{\eps,\lambda}|^{2n}\}^{\frac{1}{n}}.
\end{aligned}
\end{equation*}
So we are left to prove $\sum_{e\in\B}\E\{|\partial_e X_{\eps,\lambda}|^{2n}\}^{\frac{1}{n}}\les 1$. Since
\begin{equation*}
\partial_e X_{\eps,\lambda}=-\frac{\eps^{\frac{d}{2}+2}}{\lambda^{\frac{d}{2}+1}}\sum_{x,y\in \Z^d}\partial_e a_e \nabla G(x,e)\nabla G(y,e)f(\eps y)g(\frac{\eps x}{\lambda})
\end{equation*}
and $f,g\in\C_c(\R^d)$, by applying Proposition~\ref{p:esGr} and Lemma~\ref{l:MN} we obtain
\begin{equation*}
\begin{aligned}
\|\partial_e X_{\eps,\lambda}\|_{2n}\les &\frac{\eps^{\frac{d}{2}+2}}{\lambda^{\frac{d}{2}+1}}\sum_{x\in\Z^d}\frac{1}{|x-\underline{e}|_*^{d-1}}1_{|x|\les \frac{\lambda}{\eps}}\sum_{y\in\Z^d}\frac{1}{|y-\underline{e}|_*^{d-1}}1_{|y|\les \frac{1}{\eps}}\\
\les &\frac{\eps^{\frac{d}{2}+2}}{\lambda^{\frac{d}{2}+1}}\frac{\lambda}{\eps}\frac{1}{|\frac{\eps}{\lambda}\underline{e}|_*^{d-1}}\frac{1}{\eps}\frac{1}{|\eps e|_*^{d-1}}=\left(\frac{\eps}{\lambda}\right)^{\frac{d}{2}}\frac{1}{|\frac{\eps}{\lambda}\underline{e}|_*^{d-1}}\frac{1}{|\eps e|_*^{d-1}},
\end{aligned}
\end{equation*}
which implies 
\begin{equation*}
\sum_{e\in\B}\E\{|\partial_e X_{\eps,\lambda}|^{2n}\}^{\frac{1}{n}}\les \sum_{e\in\B}\left(\frac{\eps}{\lambda}\right)^{d}\frac{1}{|\frac{\eps}{\lambda}\underline{e}|_*^{2d-2}}\frac{1}{|\eps e|_*^{2d-2}}\les\sum_{e\in\B}\left(\frac{\eps}{\lambda}\right)^{d}\frac{1}{|\frac{\eps}{\lambda}\underline{e}|_*^{2d-2}} \les1.
\end{equation*}
The proof is complete.
\end{proof}

\appendix

\section{Technical lemmas}

\begin{lem}[vertical derivative of $G(x,y)$ with respect to $\zeta_e$]
For $e\in \B,x,y\in \Z^d,\omega\in\Omega$, we have
\begin{equation}
\partial_e G(x,y)=-\partial_e a_e\nabla G(x,e)\nabla G(y,e).
\end{equation}
\label{l:vd}
\end{lem}

\begin{proof}
Fix $e,x,y,\zeta$. By definition, the Green function $G(x,y)=\int_0^\infty q_t(x,y) \, dt$ with the heat kernel $q_t(z,y)$ solving the following parabolic problem
\begin{equation*}
\partial_t q_t(z,y)=-\nabla^*\tilde{a}(z)\nabla q_t(z,y), z\in \Z^d,
\end{equation*}
with initial condition $q_0(z,y)=1_{z=y}$. We take $\partial_e$ on both sides of the above equation to obtain
\begin{equation*}
\partial_t \partial_e q_t(z,y)=-\nabla^*\tilde{a}(z)\nabla \partial_e q_t(z,y)+\partial_e a_e\nabla q_t(e,y)(1_{z=\underline{e}}-1_{z=\bar{e}}),
\end{equation*}
with initial condition $\partial_e q_0(z,y)=0$. So $\partial_e q_t(x,y)$ is given by
\begin{equation*}
\begin{aligned}
\partial_e q_t(x,y)=&\int_0^t \sum_{z\in\Z^d} q_{t-s}(x,z)\partial_e a_e\nabla q_s(e,y)(1_{z=\underline{e}}-1_{z=\bar{e}})ds\\
=&-\partial_e a_e\int_0^t \nabla q_{t-s}(x,e)\nabla q_s(e,y)ds.
\end{aligned}
\end{equation*}
This leads to
\begin{equation*}
\begin{aligned}
\partial_e G(x,y)=\int_0^\infty \partial_e q_t(x,y)dt=&-\partial_e a_e\int_0^\infty \int_0^t \nabla q_{t-s}(x,e)\nabla q_s(e,y)dsdt\\
=&-\partial_e a_e \nabla G(x,e)\nabla G(y,e),
\end{aligned}
\end{equation*}
where we used the symmetry $q_t(x,y)=q_t(y,x)$ in the last step. The proof is complete.
\end{proof}

\begin{lem}[Spectral Gap Inequality to control fourth moment]
For any $f$ with $\E\{f\}=0$, 
\begin{equation}
\E\{f^4\}\les\left(\sum_{e\in\B} \sqrt{\E\{|f-f^e|^4\}}\right)^2.
\end{equation}
\label{l:SG}
\end{lem}

\begin{proof}
By \cite[Lemma 2]{gloria2013quantification}, if we define $\E_e\{f\}:=\E\{f|\{\zeta_{e'}\}_{e'\neq e}\}$, then
$$\E\{f^4\}\les \E\{(\sum_{e\in\B} |f-\E_e\{f\}|^2)^2\}.$$
 By expanding the right-hand side, we obtain
\begin{equation*}
\begin{aligned}
\E\{f^4\}\les &\sum_{e,e'\in\B}\E\{|f-\E_e\{f\}|^2|f-\E_{e'}\{f\}|^2\}\\
\leq &\sum_{e,e'\in\B}\sqrt{\E\{|f-\E_e\{f\}|^4\}\E\{|f-\E_{e'}\{f\}|^4\}}\\
=&\left(\sum_{e\in\B} \sqrt{\E\{|f-\E_e\{f\}|^4\}}\right)^2.
\end{aligned}
\end{equation*}
It thus suffices to show that $\E\{|f-\E_e\{f\}|^4\}\leq \frac12\E\{|f-f^e|^4\}$. In order to do so, we write $$f-f^e=f-\E_e\{f\}+\E_e\{f\}-f^e,$$ and observe that
\begin{equation*}
\begin{aligned}
\E\{|f-f^e|^4\}=&\E\{|f-\E_e\{f\}|^4\}+\E\{|\E_e\{f\}-f^e|^4\}+6\E\{|f-\E_e\{f\}|^2|\E_e\{f\}-f^e|^2\}\\
&+4\E\{(f-\E_e\{f\})(\E_e\{f\}-f^e)^3\}+4\E\{(f-\E_e\{f\})^3(\E_e\{f\}-f^e)\}.
\end{aligned}
\end{equation*}
By first averaging over $\zeta_e$ (resp.\ $\zeta_e'$), we see that the third (resp.\ fourth) term on the right-hand side is equal to zero, so the proof is complete.
\end{proof}

\begin{lem}[Sensitivity of gradient of correctors with respect to $\zeta_e$]
For $e\in \mathbb{B},x\in \Z^d$, $i,j=1,\ldots,d$ and $p\geq 1$, we have
\begin{equation}
\E\{|\nabla_j\tilde{\phi}_i(x)-\nabla_j\tilde{\phi}_i^e(x)|^p\}\les |x-\underline{e}|_*^{-pd}.
\end{equation}
\label{l:dge}
\end{lem}

\begin{proof}
By the convergence of $\nabla_j\tilde{\phi}_{\lambda,i}(x)\to \nabla_j\tilde{\phi}_i(x)$ in $L^p(\Omega)$, we only need to show 
\begin{equation*}
\E\{|\nabla_j\tilde{\phi}_{\lambda,i}(x)-\nabla_j\tilde{\phi}_{\lambda,i}^e(x)|^p\}\les |x-\underline{e}|_*^{-pd},
\end{equation*}
where the implicit multiplicative constant is independent of $\lambda$.

We write the equation satisfied by $\tilde{\phi}_{\lambda,i}$ and $\tilde{\phi}_{\lambda,i}^e$ as
\begin{eqnarray*}
\lambda \tilde{\phi}_{\lambda,i}(x)+\nabla^* \tilde{a}(x)\nabla \tilde{\phi}_{\lambda,i}(x)&=&-\nabla^*\tilde{a}(x)e_i,\\
\lambda \tilde{\phi}_{\lambda,i}^e(x)+\nabla^* \tilde{a}^e(x)\nabla \tilde{\phi}_{\lambda,i}^e(x)&=&-\nabla^*\tilde{a}^e(x)e_i.
\end{eqnarray*}
A straightforward calculation leads to
\begin{equation*}
\tilde{\phi}_{\lambda,i}(x)-\tilde{\phi}_{\lambda,i}^e(x)=\sum_{y\in\Z^d} G_\lambda(x,y)\nabla^*(\tilde{a}^e(y)-\tilde{a}(y))(\nabla \tilde{\phi}_{\lambda,i}^e(y)+e_i),
\end{equation*}
so we have
\begin{equation*}
\nabla_j\tilde{\phi}_{\lambda,i}(x)-\nabla_j\tilde{\phi}_{\lambda,i}^e(x)=\sum_{y\in\Z^d} \nabla_{x,j}\nabla_yG_\lambda(x,y)(\tilde{a}^e(y)-\tilde{a}(y))(\nabla \tilde{\phi}_{\lambda,i}^e(y)+e_i).
\end{equation*}
Since $\tilde{a}^e(y)-\tilde{a}(y)=0$ when $y\neq \underline{e}$, we conclude
\begin{equation*}
|\nabla_j\tilde{\phi}_{\lambda,i}(x)-\nabla_j\tilde{\phi}_{\lambda,i}^e(x)|\les |\nabla_{x,j}\nabla_yG_\lambda(x,\underline{e})||\nabla \tilde{\phi}_{\lambda,i}^e(\underline{e})+e_i|.
\end{equation*}
By Propositions~\ref{p:esCo} and \ref{p:esGr}, the proof is complete.
\end{proof}

\begin{lem}[Covariance estimate of $q_{ij}$]
For $i,j=1,\ldots,d$ and $x\in\Z^d$, we have
\begin{equation}
|\E\{\tilde{q}_{ij}(0)\tilde{q}_{ij}(x)\}|\les\frac{\log|x|_*}{|x|_*^d}.
\end{equation}
\label{l:covqij}
\end{lem}

\begin{rem}
Similar results in continuous setting are given in \cite[Proposition 4.7]{gu-mourrat}.
\end{rem}

\begin{proof}
By \cite[(4.4)]{gu-mourrat}, we have
\begin{equation*}
\begin{aligned}
|\E\{\tilde{q}_{ij}(0)\tilde{q}_{ij}(x)\}|=&|\mathrm{Cov}\{\tilde{q}_{ij}(0),\tilde{q}_{ij}(x)\}|\\
\les&\sum_{e\in\B}\sqrt{\E\{|\tilde{q}_{ij}(0)-\tilde{q}_{ij}^e(0)|^2\}}\sqrt{\E\{|\tilde{q}_{ij}(x)-\tilde{q}_{ij}^e(x)|^2\}}.
\end{aligned}
\end{equation*}
Recall that $q_{ij}=a_j1_{i=j}+a_jD_j\phi_i-\bar{a}1_{i=j}$, so for $e\in\B,x\in\Z^d$, we have
\begin{equation*}
|\tilde{q}_{ij}(x)-\tilde{q}_{ij}^e(x)|\les |\tilde{a}_j(x)-\tilde{a}_j^e(x)|(1+|\nabla_j\tilde{\phi}_i(x)|)+|\nabla_j\tilde{\phi}_i(x)-\nabla_j\tilde{\phi}_i^e(x)|.
\end{equation*}
By Proposition~\ref{p:esCo} and Lemma~\ref{l:dge}, we have
\begin{equation*}
\E\{|\tilde{q}_{ij}(x)-\tilde{q}_{ij}^e(x)|^2\}\les 1_{x=\underline{e}}+|x-\underline{e}|_*^{-2d},
\end{equation*}
which implies
\begin{equation*}
|\E\{\tilde{q}_{ij}(0)\tilde{q}_{ij}(x)\}|\les \sum_{e\in\B}\frac{1}{|\underline{e}|_*^d}\frac{1}{|x-\underline{e}|_*^d}\les \frac{\log|x|_*}{|x|_*^d},
\end{equation*}
where the last inequality comes from Lemma~\ref{l:cvP}. The proof is complete.
\end{proof}

\begin{lem}[Estimates on discrete convolutions]
For $\alpha,\beta>0$ with $\alpha+\beta>d$, we have
\begin{equation}
\sum_{y\in \Z^d}\frac{1}{|y|_*^\alpha}\frac{1}{|x-y|_*^\beta}\leq C_{\alpha,\beta} F_{\alpha,\beta}(x)
\end{equation}
for some constant $C_{\alpha,\beta}>0$ and
\begin{equation}
F_{\alpha,\beta}(x)=\frac{1}{|x|_*^{\alpha+\beta-d}}1_{\alpha\vee \beta<d}+\frac{1}{|x|_*^{\alpha\wedge\beta}}1_{\alpha\vee\beta>d}+\frac{\log|x|_*}{|x|_*^{\alpha\wedge\beta}}1_{\alpha\vee\beta=d}.
\end{equation}
\label{l:cvP}
\end{lem}

\begin{rem}
We will usually replace $\frac{\log|x|_*}{|x|_*^{\alpha\wedge\beta}}$ by $\frac{1}{|x|_*^{(\alpha\wedge\beta)-}}$.
\end{rem}

\begin{proof}
The proof is standard. Since $\alpha+\beta>d$, we only need to consider the region $|x|>100$.

For fixed $x$, let $I_1=\{y:|y|\leq |x|/2\}$, $I_2=\{y: |y-x|\leq |x|/2\}$, and $I_3=\Z^d\setminus (I_1\cup I_2)$. We control the sum in each region separately. The proof for each case is similar and we only use the following two facts:
\begin{itemize}
\item $|x|_*\les |y-x|_*$ in $I_1$, $|x|_*\les |y|_*$ in $I_2$ and $|y-x|_*\sim |y|_*$ in $I_3$,
\item for any $\gamma>0$, $\sum_{|y|\leq |x|}|y|_*^{-\gamma}\les |x|_*^{d-\gamma}1_{\gamma<d}+\log|x|_*1_{\gamma=d}+1_{\gamma>d}$.
\end{itemize}

In $I_3$, we have 
\begin{equation*}
\sum_{I_3}\frac{1}{|y|_*^\alpha}\frac{1}{|x-y|_*^\beta}\les \sum_{|y|\geq |x|/2}\frac{1}{|y|_*^{\alpha+\beta}}\les  \frac{1}{|x|_*^{\alpha+\beta-d}}.
\end{equation*}

If $\alpha\vee\beta<d$, the discussion for $I_1$ and $I_2$ are the same. Take $I_1$ for example, we have 
\begin{equation*}
\sum_{I_1}\frac{1}{|y|_*^\alpha}\frac{1}{|x-y|_*^\beta}\les \frac{1}{|x|_*^\beta}\sum_{|y|\leq |x|/2}\frac{1}{|y|_*^\alpha}\les \frac{1}{|x|_*^{\alpha+\beta-d}}.
\end{equation*}

If $\alpha\vee\beta>d$, we assume $\alpha\geq \beta$, so $\alpha>d$. By similar discussion, in $I_2$ we get the estimate
\begin{equation*}
\sum_{I_2}\frac{1}{|y|_*^\alpha}\frac{1}{|x-y|_*^\beta}\les\frac{1}{|x|_*^{\alpha+\beta-d}}1_{\beta<d}+\frac{\log |x|_*}{|x|_*^\alpha}1_{\beta=d}+\frac{1}{|x|_*^\alpha}1_{\beta>d}.
\end{equation*}
 In $I_1$, we have $\sum_{|y|\leq |x|/2}|y|_*^{-\alpha}|x-y|_*^{-\beta}\les |x|_*^{-\beta}$. So an overall bound is given by $|x|_*^{-\beta}=|x|_*^{-\alpha\wedge \beta}$.

If $\alpha\vee \beta=d$, we assume $\alpha=d$. If $\beta=d$, in both $I_1,I_2$ we get the bound $|x|_*^{-d}\log|x|_*$. If $\beta<d$, in $I_1$ we get a bound $|x|_*^{-\beta}\log|x|_*$ and in $I_2$ we get the bound $|x|_*^{d-\alpha-\beta}=|x|_*^{-\beta}$. Therefore, an overall bound is $|x|_*^{-\beta}\log|x|_*=|x|_*^{-\alpha\wedge\beta}\log|x|_*$.

The proof is complete.
\end{proof}

\begin{lem}
Let $x_1,\ldots,x_k\in \Z^d$ be mutually different, and for each $i$, let $\bar{i}$ be such that $|x_{\bar{i}}-x_i|=\min_{j\neq i}|x_j-x_i|$. Assume $\alpha_1,\ldots,\alpha_k\in (0,d)$ and $\alpha_i+\alpha_j>d$ for $i\neq j$, then
\begin{equation}
\sum_{y\in \Z^d} \prod_i \frac{1}{|y-x_i|_*^{\alpha_i}}\les \sum_{i=1}^k\prod_{j\neq i,\bar{i}}\frac{1}{|x_j-x_i|_*^{\alpha_j}} \frac{1}{|x_i-x_{\bar{i}}|_*^{\alpha_i+\alpha_{\bar{i}}-d}}.
\end{equation}
\label{l:cvmP}
\end{lem}

\begin{proof}
For each $i$, we define the region $I_i=\{y: |y-x_i|\leq \min_j|y-x_j|\}$, i.e., the set of points that are closest to $x_i$. If $y\in I_i$, we have $|y-x_j|\geq |x_i-x_j|/2$ for any $j\neq i$. Therefore,
\begin{equation*}
\sum_{y\in I_i} \prod_j \frac{1}{|y-x_j|_*^{\alpha_j}}\les \prod_{j\neq i,\bar{i}}\frac{1}{|x_j-x_i|_*^{\alpha_j}}\sum_{y\in \Z^d} \frac{1}{|y-x_i|_*^{\alpha_i}}\frac{1}{|y-x_{\bar{i}}|_*^{\alpha_{\bar{i}}}}.
\end{equation*}
Since $\alpha_i+\alpha_{\bar{i}}>d$, the sum over $y$ can be bounded using Lemma~\ref{l:cvP}, e.g.,  when $\alpha_i<d$ for all $i$, we have
\begin{equation*}
\sum_{y\in I_i} \prod_j \frac{1}{|y-x_j|_*^{\alpha_j}}\les
\prod_{j\neq i,\bar{i}}\frac{1}{|x_j-x_i|_*^{\alpha_j}} \frac{1}{|x_i-x_{\bar{i}}|_*^{\alpha_i+\alpha_{\bar{i}}-d}}.
\end{equation*}
The proof is complete.
\end{proof}

\begin{rem}
From the proof of Lemma~\ref{l:cvmP}, we see that the condition $\alpha_i+\alpha_j>d$ for all $i\neq j$ is not necessary to obtain similar estimates. For example, for each $i$, as long as we can find $j\neq i$ such that $\alpha_i+\alpha_j>d$, the integral in $I_i$ can be controlled by a similar bound. 
\label{r:cvmP}
\end{rem}

Recall that the error function $\cE$ in Proposition~\ref{p:exCov} is given by
\begin{equation}
\begin{aligned}
\mathcal{E}(x_1,x_2,x_3,x_4)=\sum_{e\in \mathbb{B}}\sum_{i=1}^4\frac{\log |\underline{e}-x_i|_*}{|\underline{e}-x_i|_*^d}\prod_{j=1, j\neq i}^4
\frac{1}{|\underline{e}-x_j|_*^{d-1}}\\
\les \sum_{v\in \Z^d}\sum_{i=1}^4\frac{1}{|v-x_i|_*^{d-}}\prod_{j=1, j\neq i}^4
\frac{1}{|v-x_j|_*^{d-1}}.
\label{bdCor}
\end{aligned}
\end{equation}
By using Lemmas~\ref{l:cvP} and \ref{l:cvmP}, we have the following control on the error function:
\begin{lem}[Estimation of $\cE(x,y,z,w)$]
Let $x,y,z,w\in \Z^d$, 
\begin{itemize}
\item if $x=y=z=w$, $\cE(x,y,z,w)\les 1$.
\item if $x=y=z\neq w$, $\cE(x,y,z,w)\les |x-w|_*^{1-d}$.
\item if $x=y\neq z=w$, $\cE(x,y,z,w)\les |x-w|_*^{2-2d}$.
\item if $x=y$ and $y,z,w$ are mutually different, let $S=\{x,z,w\}$, 
\begin{equation*}
\cE(x,y,z,w)\les\sum_{v\in S}\prod_{u\in S\setminus\{v\}}\frac{1}{|u-v|_*^{d-1}}.
\end{equation*}
\item if $x,y,z,w$ are mutually different, let $S=\{x,y,z,w\}$,
\begin{equation*}
\cE(x,y,z,w)\les\sum_{v\in S} \prod_{u\in S\setminus \{v\}} \frac{1}{|u-v|_*^{(d-1)-}}.
\end{equation*}
\end{itemize}
\label{l:erCov}
\end{lem}

\begin{proof}
The proofs of different cases are similar. We only discuss the case when $x,y,z,w$ are mutually different. By \eqref{bdCor}, we consider a term of the form:
\begin{equation*}
\sum_{v\in \Z^d}\frac{1}{|v-x|_*^{d-}}\frac{1}{|v-y|_*^{d-1}}\frac{1}{|v-z|_*^{d-1}}\frac{1}{|v-w|_*^{d-1}}.
\end{equation*}
Recall that $S=\{x,y,z,w\}$. For $u\in S$, let $I_u=\{v\in \Z^d:|v-u|=\min_{q\in S}|v-q|\}$. The proof is the same as in Lemma~\ref{l:cvmP}, so we do not provide all details.

For $I_x$, we have
\begin{equation*}
\begin{aligned}
&\sum_{v\in I_x}\frac{1}{|v-x|_*^{d-}}\frac{1}{|v-y|_*^{d-1}}\frac{1}{|v-z|_*^{d-1}}\frac{1}{|v-w|_*^{d-1}}\\
\les &\frac{1}{|x-z|_*^{d-1}}\frac{1}{|x-w|_*^{d-1}}\sum_{v\in \Z^d}\frac{1}{|v-x|_*^{d-}}\frac{1}{|v-y|_*^{d-1}}\\
\les&\prod_{u\in S\setminus\{x\}}\frac{1}{|x-u|_*^{(d-1)-}}.
\end{aligned}
\end{equation*}
Note that we replaced $\frac{1}{|x-z|_*^{d-1}}\frac{1}{|x-w|_*^{d-1}}$ by $\frac{1}{|x-z|_*^{(d-1)-}}\frac{1}{|x-w|_*^{(d-1)-}}$.

For $I_y$, we have
\begin{equation*}
\begin{aligned}
&\sum_{v\in I_y}\frac{1}{|v-x|_*^{d-}}\frac{1}{|v-y|_*^{d-1}}\frac{1}{|v-z|_*^{d-1}}\frac{1}{|v-w|_*^{d-1}}\\
\les &\frac{1}{|y-z|_*^{d-1}}\frac{1}{|y-w|_*^{d-1}}\sum_{v\in \Z^d}\frac{1}{|v-x|_*^{d-}}\frac{1}{|v-y|_*^{d-1}}\\
\les&\prod_{u\in S\setminus\{y\}}\frac{1}{|y-u|_*^{(d-1)-}}.
\end{aligned}
\end{equation*}
The sums in $I_z,I_w$ are discussed in the same way. The proof is complete.
\end{proof}

\begin{lem}
For $x\in \Z^d$, $p>0$ and $\eps\in (0,1)$,
\begin{equation}
\sum_{y\in\Z^d}\frac{1}{|x-y|_*^{d-1}}1_{|y|\les \eps^{-1}} \les \frac{\eps^{-1}}{|\eps x|_*^{d-1}},
\end{equation}
and
\begin{equation}
\sum_{y\in\Z^d}\frac{1}{|x-y|_*^d}\frac{1}{|\eps y|_*^p}\les \frac{|\log\eps|}{|\eps x|_*^{d\wedge p}}.
\end{equation}
\label{l:MN}
\end{lem}
We refer to \cite[Lemmas 2.2 and 2.3]{MN} for a proof.

\bigskip

\noindent \textbf{Acknowledgements.} We thank Scott Armstrong with whom the ideas of Section~1 were developped, and who accepted to let us present them here. We thank Antoine Gloria and Wenjia Jing for helpful discussions.

\def\cprime{$'$}

\end{document}